\documentclass{article}

\usepackage{arxiv}

\usepackage[utf8]{inputenc} 
\usepackage[T1]{fontenc}    
\usepackage{hyperref}       
\usepackage{url}            
\usepackage{booktabs}       
\usepackage{amsfonts}       
\usepackage{nicefrac}       
\usepackage{microtype}      
\usepackage{lipsum}
\usepackage{graphicx}
\usepackage{amsmath,amssymb}
\usepackage{bbm, dsfont}
\usepackage{ulem}
\usepackage{multirow}
\usepackage{tocbibind}

\usepackage[toc,page]{appendix}

\newcommand{\Z}{\mathbb{Z}}
\newcommand{\R}{\mathbb{R}}
\newcommand{\E}{\mathbb{E}}

\graphicspath{ {./images/} }
\newtheorem{theorem}{Theorem}
\newtheorem{assumption}{Assumption}
\newtheorem{proposition}{Proposition}
\newtheorem{lemma}{Lemma}
\newtheorem{corollary}{Corollary}
\title{Data-driven Ranking and Selection under Input Uncertainty}

\author{
 Di Wu \\
  Amazon Web Services\\ 
  Seattle, WA 98109\\
  \texttt{woody10074026@gmail.com}
   \And
 Yuhao Wang \\
  School of Industrial and Systems Engineering\\
  Georgia Institute of Technology\\
  Atlanta, GA 30332\\
  \texttt{yuhaowang@isye.gatech.edu}
  \And
 Enlu Zhou \\
  School of Industrial and Systems Engineering\\
  Georgia Institute of Technology\\
  Atlanta, GA 30332\\
  \texttt{enlu.zhou@isye.gatech.edu}
}

\begin{document}
\maketitle

\begin{abstract}
We consider a simulation-based Ranking and Selection (R\&S) problem with input uncertainty, where unknown input distributions can be estimated using input data arriving in batches of varying sizes over time. Each time a batch arrives, additional simulations can be run using updated input distribution estimates. The goal is to confidently identify the best design after collecting as few batches as possible. We first introduce a moving average estimator for aggregating simulation outputs generated under heterogenous input distributions. Then, based on a Sequential Elimination framework, we devise two major R\&S procedures by establishing exact and  asymptotic confidence bands for the estimator. 
We also extend our procedures to the indifference zone setting, which helps save simulation effort for practical usage. Numerical results show  the effectiveness  and necessity  of our procedures in  controlling error from input uncertainty. Moreover, the efficiency can be further boosted through optimizing the ``drop rate'' parameter,  which is the proportion of past simulation outputs to discard,  of the moving average estimator.


\end{abstract}

%



\maketitle

%


\section{Introduction} \label{sec:intro}
Stochastic simulation has been used widely for evaluating and optimizing complex systems arising in manufacturing, transportation, and many other domains. Building a stochastic simulation model requires to account for external random factors (e.g., lead time, traveling time, demand load) that affect the system's performance. A common practice is to model such randomness by probability distributions, from which random samples are generated to simulate real-world scenarios. These distributions are referred to as ``input distributions''.

In practice, input distributions need to be estimated from finite historical data known as ``input data'', where the estimation error results in \emph{input uncertainty} (IU). Aside from this extrinsic uncertainty, the simulation output is also subject to \emph{stochastic uncertainty} (SU), which is the intrinsic uncertainty induced by random samples generated from input distributions. However, these two types of uncertainty differ in nature: while SU can be reduced by investing simulation effort, IU is determined by the availability of input data which is largely beyond a modeler's control. Ignoring IU can be risky when simulation results are used to inform decisions (see, e.g., \cite{zhou2015simulation}).  

In many application problems, input data are collected frequently to reduce the IU of the simulation model. The simulation model, which is often time  consuming to run, is then used to compare different designs/strategies to find out the best one with high confidence. Such a procedure is often referred to as Ranking and Selection (R\&S). Data often come on a fast time scale (such as daily), but system designs cannot be changed so frequently; otherwise, it not only incurs excessive cost and labor but also might cause instability to the system. Hence, simulation becomes extremely useful to test and compare potential designs/strategies before they are  implemented in the real system. The following application examples illustrate such a problem setting.

\begin{enumerate}
\item \textbf{Start ride sharing service in an airport}. In some cities, ride sharing used to be regulated (disallowed) in the airport. When such regulation is lifted, the platform will face a cold-start problem for pricing and matching strategies due to the lack of historical data at the airport (such as driver and rider arrival rates, pricing and earnings elasticities) which are the inputs to a typical simulation system that platform uses to choose the pricing and matching strategies. Such simulation can take hours per replication to capture the day/week effect in the market condition. As new data come in daily, the simulation input will be updated and new simulations are run to compare potential pricing and matching strategies until one best strategy is identified and then rolled out to the platform.

\item {\bf Supply chain optimization}. A multinational retailer makes use of a complex supply chain simulation model to evaluate and compare a few potential inventory policies (including the current policy in use). The model’s input distributions capture the uncertainty in production lead time, transit lead time, demand, and so on. As data accrues daily over the selling season, the retailer continuously updates the input models and runs simulations until it identifies the best inventory policy with high confidence and then implements the policy on the real system. 

\item \textbf{Budget R\&D capital}: A company conducts  testing  on a set of products, with the goal to find the product with the highest expected net gain within the shortest possible time period. Testing each product is time costly and subject to stochastic error, so multiple replications of testing are needed. As outside investments reveal over time, the company uses this data to update its estimate of the market condition of each product with reduced uncertainty and runs new tests until it narrows down to one best product and then releases it to the market. 
\end{enumerate}

These examples motivate us to consider a fixed confidence ranking and selection (R\&S) problem where new input data arrives over time in batches of possibly varying and random sizes. In each time stage, the number of simulation replications that can be run is limited, due to the expensive simulation cost  and the time length of each stage. On the one hand, we can use the new data batch to update the input distribution of the simulation model so as to reduce IU. Because of the limited simulation replications at each time stage, it is necessary to aggregate the simulation outputs across different time stages to reduce SU in the performance estimate. On the other hand, simulations are run under a different (updated) input distribution at each stage, and hence the simulation outputs are correlated and differently distributed  across different time stages. This creates a major challenge in designing a R\&S procedure for this setting, since most classical R\&S requires independent and identically distributed (i.i.d.) condition on the simulation outputs. 
Further complications lie in how to (i) update the input models; and (ii) aggregate simulation outputs generated under heterogenous input models. While there are well-established methods for these two tasks, whether they would fit into the R\&S framework needs investigation. 

To address these challenges in the problem setting above, we propose a moving average estimator for system performance to aggregate the simulation outputs across time stages. Intuitively, the moving average estimator drops the obsolete simulation outputs while keeping the more recent outputs that are generated under closer input distributions to the latest one. A parameter, called drop rate, is used to balance the trade-off between bias (due to keeping the old simulation outputs) and variance (due to the limited number of simulation outputs) in the performance estimate. We then build on the sequential elimination (SE) framework, which was first developed by \cite{even2002pac,even2006action}, to perform R\&S  using the moving average estimator. Specifically, by computing the confidence bands to account for both IU and SU of the moving average estimates, the SE procedures eliminate one (statistically) inferior design each time. We summarize the contributions of this paper as follows.

\begin{enumerate}
\item
This paper along with our earlier conference paper, \cite{wu2019fixed}, are the first to consider streaming input data in R\&S problems and design a data-driven approach. The simulation outputs generated from different stages are neither independent nor identically distributed, which is a major theoretical challenge that makes our approach drastically different from the classical R\&S procedures.
\item
We introduce a moving average estimator to aggregate simulation outputs generated under heterogenous input models, and substantially extend a Sequential Elimination framework to handle IU with streaming input data. 
\item
We design Sequential Elimination procedures, called SEIU and SEIU-MCB, based on confidence bands established using two different approaches. The SEIU approach relies on exact confidence bands but tends to be conservative; the SEIU-MCB approach relies on  asymptotically valid  but tighter confidence bands. Specifically, the latter leverages  results from ``Multiple Comparison with the Best(MCB)" \cite{chang1992optimal} and an asymptotic normality result that we establish for characterizing the trade-off between IU and SU under streaming data, which is of independent interest and can be viewed as a multi-stage generalization of a well-known result in \cite{cheng1997sensitivity}. We also extend the two SE procedures to the indifference zone setting, named as SEIU(IZ) and SEIU-MCB(IZ), which further boosts the procedures for practical usage. 
\item
The necessity and effectiveness of the procedures are demonstrated numerically through a simple quadratic problem and a more sophisticated production-inventory problem. Furthermore, we show that SEIU-MCB can be accelerated by optimizing the drop rate parameter that shows up in the moving average estimator.
\end{enumerate}

We should remark that, the aggregation of simulation outputs under different input distributions via the likehihood ratio method for has been studied by \cite{feng2015green,feng2017green,eckman2018green,feng2019efficient} under the name ``green simulation'', and recently 
by \cite{9384032} for simulation optimization. Such methods have the potential to be extended to the setting in this paper. We leave the adaptation of these methods to data-driven R\&S as an open problem.

\subsection{Literature Review}
The mixed effect of IU and SU on simulation output has been studied in a variety of contexts. In terms of quantifying the impact of IU on a single system design's simulation output, the earliest method at least dates back to \cite{barton1993uniform}, 
followed by many other works, including but not limited to  \cite{cheng1997sensitivity}, \cite{chick2001input,zouaoui2003accounting,zouaoui2004accounting, xie2014bayesian}, \cite{barton2014quantifying}, \cite{xie2014bayesian,xie2016multivariate}, \cite{lam2021subsampling}, \cite{lin2015single}, \cite{lam2017empirical}, \cite{feng2019efficient}, \cite{song2015quickly},\cite{zhu2020risk}. 
These works assume a fixed batch of input data, but recently \cite{Zhou2018} and \cite{liu2019online}  considered streaming input data, the same setting as this paper. However, they take the likelihood ratio method to estimate the performance and a sampling-based method to quantify IU, whereas we develop a moving average estimator and analytically compute confidence bands to quantify IU. We refer the reader to \cite{corlu2020stochastic} for a recent review on the topic of input uncertainty.

There is an abundant and fast growing R\&S literature. Although a comprehensive review is out of this paper's scope, in what follows we make our best effort to present our work through a broader perspective. 
In simulation optimization or ordinal optimization (\cite{ho1992ordinal}), R\&S arises in the context of identifying the best system design (or a subset of good designs) through noisy simulation outputs. Simulation-based R\&S is predominantly studied under the assumption that SU is the only source of uncertainty. Research in this field can roughly be categorized into two problem settings: fixed budget and fixed confidence. In the fixed budget setting, the total number of simulation runs is constrained, and the goal is to maximize the probability of correct selection (PCS) of the best design. In this paper, we focus on the fixed confidence setting, where the goal is to attain a pre-specified PCS using as little simulation effort as possible.

The study on fixed confidence R\&S is primarily focused on the Indifference Zone (IZ) formulation, where the goal is to select the best system design with a target probability when the top two designs differ by at least some value $\delta$ in expected performance. 
Since \cite{bechhofer1954single}, the IZ formulation has been been studied actively by the statistics community. In simulation literature, the milestone work that adapts IZ to simulation-based R\&S is \cite{kim2001fully}, in which the proposed KN procedure not only outperforms classical statistical procedures, but also allows the use of Common Random Numbers, a variance reduction technique, for further enhancement. The KN procedure has been extended to steady-state simulation in \cite{goldsman2002ranking} and \cite{kim2006asymptotic}, and also inspired numerous subsequent works on fully sequential procedures including \cite{batur2006fully,hong2005tradeoff,hong2007selecting,pichitlamken2006sequential,jeff2006fully} and so on. More recently, \cite{frazier2014fully} proposes an IZ procedure which is Bayesian in spirit but provides guarantee on the frequentist PCS; \cite{luo2015fully} designs procedures for large-scale R\&S problems in parallel computing environments; \cite{fan2016indifference} develops procedures which do not require an IZ parameter. 

The fixed confidence R\&S problem has also been studied extensively in the Multi-Armed Bandits (MAB) literature under the name of Best Arm Identification (BAI). 
Except for minor difference in performance criterion,
R\&S and BAI are essentially the same mathematical problem. Nevertheless, the research in the two fields diverges in several aspects, with the most notable difference that BAI focuses more on characterizing complexity results and designing algorithms that match some worst-case or problem-specific lower bounds, such as  
the seminal works of \cite{even2002pac,mannor2004sample}, followed by \cite{gabillon2012best,karnin2013almost,jamieson2014best,kaufmann2016complexity}, and culminating in \cite{garivier2016optimal}. 
Aside from these frequentist results, \cite{russo2016simple} proposes some simple algorithms from a Bayesian viewpoint.

Compared with the studies on classical R\&S and BAI which have gradually matured over the past few decades, research on R\&S under IU is only starting to gain momentum. One stream of work takes a distributionally robust optimization approach by assuming that the true input distribution is contained in a finite set of known distributions (i.e., ambiguity set), with a goal of selecting the design with the best worst-case performance over the ambiguity set, such as \cite{gao2017robust,xiao2018simulation,xiao2020optimal,fan2020distributionally, wu2017OCBAIU}.  
Another line of research aims to screening out as many inferior designs as possible in the presence of IU, such as \cite{corlu2013subset,corlu2015subset, song2019input}. 
Our work has a close connection with \cite{song2019input} in that we both assume parametric input distributions and incorporates the (MCB) framework from \cite{chang1992optimal}. However, despite these similarities,  our setting of streaming input data is fundamentally different. In \cite{song2019input} the performance estimate only consists of i.i.d. samples simulated under the same input distribution, whereas in this paper we sequentially update input distributions over time and aggregate past simulation outputs from heterogeneous input distributions. Therefore, the methodologies therein cannot be easily extended to our problem. 

All the papers on R\&S under IU listed above share one assumption in common:  input data is a static dataset which will not grow over time. Streaming data have only been considered recently  for R\&S by \cite{wu2019fixed}, which is a preliminary conference version of this paper. \cite{song2019stochastic}, \cite{liu2021Bayesian}, and \cite{liu2022bayesian} studied continuous simulation optimization with streaming data.

The rest of the paper is organized as follows. We formulate the problem in Section \ref{sec:formulation}, and extend the Sequential Elimination framework in Section \ref{sec:SE}. We derive the exact and asymptotically valid confidence bands in the SE framework and present the corresponding procedures SEIU and SEIU-MCB in Sections \ref{sec:exact} and \ref{sec:SEIU-MCB}, respectively. We extend the procedures to the IZ setting in Section \ref{sec:indifference zone}. Finally, we numerically demonstrate the procedures' performance in Section \ref{sec:numerical}, and conclude in Section \ref{sec:conclusion}.

\section{Problem Formulation} \label{sec:formulation}
\subsection{Review on fixed confidence R\&S without IU}
We begin with a brief review on the classical fixed confidence R\&S problem and lay down some basic notations. Suppose that we are given a set of designs $\mathcal{I} =\{1, \ldots, K\}$, and the goal is to find the design with the highest expected performance. A design $i \in \mathcal{I}$ will be evaluated through repeated simulation runs, where during each run we first generate a random sample $\xi$ from a distribution $P_i$, and then compute the output by evaluating a deterministic function of $i$ and $\xi$.

The input distributions $\{P_i\}$  are used as input to the simulation model to capture various sources of real-world randomness. In classical R\&S literature, $\{P_i\}$ are assumed to be known, and the simulation outputs for any design $i$, denoted by $\{X_{i,r}\}_r$, are independent and identically   (i.i.d.). Let $\mu_i := \E_{P_i}[X_{i,1}]$ be the true mean performance of design $i$. Then, $\mu_i$ can be estimated by averaging the simulation outputs $\{X_{i,r}\}_r$. Due to finite-sample error, the probability of selecting the best design (Probability of Correct Selection, or PCS) is often adopted to quantify the confidence of selection. In the fixed confidence setting, there is no constraint on the simulation budget, but the PCS is required to exceed a pre-specified level. We refer to procedures that guarantee to attain the PCS target as \emph{statistically valid}. The core of the fixed confidence R\&S problem is to find statistically valid procedures which terminate after as few simulation runs as possible.

\subsection{Fixed confidence R\&S with streaming input data}
A limitation of the classical R\&S framework is that true input distributions are rarely known in practice; instead, they must be estimated using finite data, which incurs input uncertainty (IU) that propagates to simulation output.
Throughout the paper, we assume that all the designs share the same set of input distributions. Furthermore, to facilitate the characterization of IU, we impose the following structure on the input distributions. 

\begin{assumption} \label{assump:parametric}
The set of input distributions belongs to a known parametric family, which contains $S$ mutually independent distributions with unknown parameters $ \{\theta_1^c, \theta_2^c,\cdots,\theta_S^c\}$, where $\theta_s^c \in \mathbb{R}^{d_s}$ .
\end{assumption}
The parametric assumption is common in the literature of R\&S. Let $d:=\sum_{s=1}^S d_s$ be the total dimension of parameters, and $\theta := [\theta_1^\intercal, \ldots, \theta_S^\intercal]^\intercal$ be a vector in $\R^d$ concatenating all parameters. The input distributions then form a product measure $P_\theta = \prod_{s=1}^S P_{\theta_s}$. Also let $\theta^c \in \R^d$ denote the true input parameter. Under Assumption \ref{assump:parametric}, the expected performance of design $i$ becomes a function of $\theta$ and is therefore denoted $\mu_i(\theta)$. 

The negative impact of IU on R\&S can be seen as follows. For simplicity, assume throughout the paper that $b \in \mathcal{I}$ is the unique true best design, i.e., $\mu_b(\theta^c) > \max_{i\neq b}\mu_i(\theta^c)$. When $\theta^c$ is unknown and is replaced by a finite-sample estimate $\hat{\theta}$, the best design under $\hat{\theta}$ may be different from $b$, in which case one cannot correctly identify the true best design even using infinite simulation runs. If input data is given as a static dataset which does not grow over time, then our best hope is to gauge the impact of IU and provide a more conservative statistical guarantee (e.g., a lower PCS) due to limited resolution in performance estimation.

Assuming a static input dataset is reasonable in applications where gathering data is expensive, labor-intensive, or time-consuming. Nevertheless, with the advancement in big data technology, additional data can sometimes be collected efficiently and economically. This motivates us to consider a R\&S problem with streaming input data, which we elaborate on as follows. 

Suppose that for each input distribution $P_{\theta^c_s}$, the input data consists of i.i.d. samples denoted by $\{\zeta_{s,1}, \zeta_{s,2}, \ldots\}$. The input data arrives in batches sequentially in time, where the sample size of batch $t$ is $n_{s}(t) \ge 1$. Denote by $N_s(t) :=\sum_{\ell=1}^t n_s(\ell)$ the total number of data samples up to batch $t$, and let $N_s(0) = 0$. We allow $n_s(t)$ to vary in $s$ and $t$ in order to have the most flexible model for multiple streams of input data with time-varying batch sizes.

The arrival process of input data thus divides the simulation process into multiple time stages. At the beginning of stage $t$, a new batch of input data arrives, which, together with historical data, is used to compute an updated parameter estimate $\hat{\theta}_t$. After that, incremental simulations are run by drawing i.i.d. samples from the updated input distributions $P_{\hat{\theta}_t}$. Similarly, let $m(t) \ge 1$ be the number of incremental simulation runs for each design at stage $t$, and let $M(t) := \sum_{\ell=1}^t m(\ell)$ be the total number of runs up until stage $t$. The simulation batch size $m(t)$ can vary across different stages, since the inter-arrival times between two adjacent data batches could be different. However, here we require the same batch size of simulation runs for each design because we will make pairwise comparison, which uses the Common Random Number (CRN)  strategy.

\begin{figure}[t] 
\centering
\includegraphics[width=0.95\linewidth]{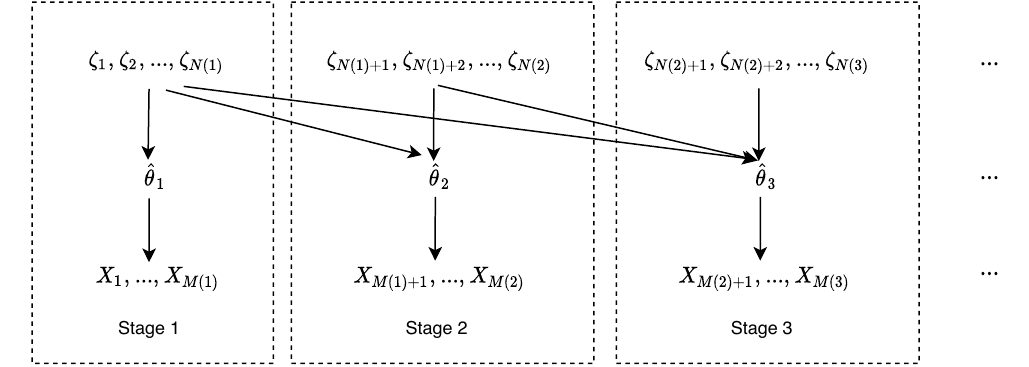}
\caption{Illustration of simulating a single design with streaming input data, where $\zeta$ denotes input data, $\hat{\theta}$ is the input parameter estimate, and $X$ denotes simulation output.}%
\label{fig:illustrate}%
\end{figure}

Figure \ref{fig:illustrate} illustrates how simulation outputs are generated for a single design, where we highlight the following observations. First, the parameter estimates $\{\hat{\theta}_t\}$ are correlated since they are computed using the same stream of input data. Second, since $\{\hat{\theta}_t\}$ are random variables that generally take different values, the simulation outputs $X$ are not identically distributed across different stages. Third, conditioned on $\{\hat{\theta}_t\}$, the outputs $X$ are i.i.d. within the same stage and independent across stages, but unconditional independence no longer holds since $X$ are jointly affected by the correlation among $\{\hat{\theta}_t\}$. In sum, the simulation outputs are neither independent nor identically distributed, which sets the problem apart from most of the R\&S studies.

The online estimation of an individual design's performance under streaming input data is an important problem in its own right.  In this paper, however, we are concerned with designing R\&S procedures that can achieve a pre-specified PCS upon termination. To make the problem well-defined, we first need to address the following questions.

\begin{enumerate}
\item[(i)]
What is the estimator of the input parameter $\theta^c$?
\item[(ii)]
How to aggregate non-i.i.d. simulation outputs to estimate $\mu_i(\theta^c)$?
\end{enumerate}

Many existing methods can be applied to these two problems, but most of them do not suit our purpose. The key is to find an estimator which would allow a tractable decomposition of the correlation among different stages' simulation outputs. This motivates us to make the following assumption.

\begin{assumption} \label{assump:estimator}
For each input parameter $\theta^c_s$, there exists a function $D_s$ such that 
\begin{equation} \label{eq:theta-est}
\hat{\theta}_{s,t} := \frac{1}{N_s(t)} \sum_{j=1}^{N_s(t)} D_s(\zeta_{s,j})
\end{equation}
is an unbiased estimator of $\theta^c_s$.
\end{assumption}

Assumption \ref{assump:estimator} can often be satisfied through reparametrization. Specifically, for a parametric family of
distributions with $k$ unknown parameters, the parameters (or some transformation of the parameters, which can be transformed back to the parameters) can often be estimated by the first $k$ moments with $D_s(x)$ being the
functions $x,x^2,\cdots,x^k$. For example, if we reparameterize the normal distribution by the first two moments, then $D_s(\zeta) = [\zeta, \zeta^2]^\intercal$ satisfies the desired property. Let $D_{s,j}$ be a shorthand for $D_s(\zeta_{s,j})$. Then, $\{D_{s,j}\}$ are i.i.d. transformed input data samples with $\E[D_{s,1}] = \theta^c$. As we shall see in Section \ref{sec:SEIU-MCB}, the additive form of $\hat{\theta}_{s,t}$ plays an important role in designing R\&S procedures.

The other question is how to estimate $\mu_i(\theta^c)$ using samples generated under different input distributions. To simplify indexing, let $\{X_{i,r}(\hat{\theta}_t)\}_{r=1}^{m(t)}$ denote the batch of simulation outputs generated at stage $t$ under the input distribution $P_{\hat{\theta}_t}$. We consider the following moving average estimator,
\begin{equation} \label{eq:mae}
\hat{\mu}_{i,t} := [M(t)- M(t_\eta)]^{-1} \sum_{\ell=t_\eta+1}^t \sum_{r=1}^{m(\ell)} X_{i,r}(\hat{\theta}_\ell),
\end{equation}
where $\eta \in [0,1)$ is called the \emph{drop rate} which controls the amount of effective samples used for computation, and $t_\eta := \lfloor \eta t \rfloor$ is the number of stages ``discarded''. In words, the estimator in (\ref{eq:mae}) only averages the latest $(1-\eta)$ portion of all the simulation output, hence the name ``moving average''. The idea of throwing away some earlier samples is motivated by the following two extreme cases.
\begin{enumerate}
\item[(i)]
$\eta=0$: Keeping all the outputs helps reduce SU but also retains all the biases $\{\mu_i(\hat{\theta}_t) - \mu_i(\theta^c)\}$ that accumulate over time.
\item[(ii)]
$\eta=1$: This roughly corresponds to keeping only the latest output $X_{i,m(t)}(\hat{\theta}_t)$, which reduces the bias but also prohibits the estimator from converging to $\mu_i(\theta^c)$.
\end{enumerate}

In Section \ref{sec:SEIU-MCB}, we explicitly characterize this tradeoff between IU and SU through $\eta$, and show that optimzing $\eta$ can lead to a lower asymptotic variance than in both of the two extreme cases. 

We now formally state the problem as follows. We aim to find the best design among $K$ alternatives via simulation. The designs share the same input distribution with an unknown parameter $\theta^c$, which can be estimated via (\ref{eq:theta-est}) using input data arriving in batches. The arrival process divides time into multiple stages, where during each stage we can update the input parameter estimate $\hat{\theta}$ and run incremental simulations under $P_{\hat{\theta}}$. The performance of each design is estimated via a moving average estimator in (\ref{eq:mae}). Given $\alpha \in (0,1)$, the goal is to design a procedure which terminates after a number of stages and outputs the true best design with probability $\ge 1-\alpha$.

On a side note, classical R\&S procedures may still be applied to our problem. For instance, we can collect a few batches of input data, and apply KN as if the input distribution estimates are accurate. However, statistical validity is no longer guaranteed under the influence of IU. While in theory this issue can be overcome with sufficient data, classical procedures have no way of telling when to safely stop collecting data.

\section{Sequential Elimination Framework} \label{sec:SE}
In fixed confidence R\&S, an important idea underlying many sequential procedures is to construct a ``continuation region'', where a design can be confidently labeled as suboptimal if its performance estimate exits the region. One example is the triangular region in the KN procedure (see \cite{kim2001fully}). In our problem setting, most of the existing methods are inapplicable since they rely on key assumptions such as i.i.d. and normality. To construct a continuation region in the presence of IU, we resort to a Sequential Elimination (SE) framework in \cite{even2002pac,even2006action}.

In short, an SE procedure sequentially eliminates inferior designs until there is only one left, and the surviving design is output as the estimated best design. Let $\delta_{ij}(\theta) := \mu_i(\theta) - \mu_j(\theta)$ and $\hat{\delta}_{ij,t} = \hat{\mu}_{i,t} - \hat{\mu}_{j,t}$. A requirement for an SE procedure is a collection of the confidence bands for $\delta_{ij}(\theta^c)$,
\begin{equation} \label{eq:pairwise-bands}
\mathbb{P}\left\{\bigcap_{t=1}^\infty \bigcap_{\substack{i<j} }  \left\{|\hat{\delta}_{ij,t} - \delta_{ij}(\theta^c)| \le w_{ij, t} \right\} \right\} \ge 1-\alpha, 
\end{equation}
where $w_{ij,t} \rightarrow 0$ as $t \rightarrow \infty$.
The elimination happens to design $i$ if there exists a design $j$ such that
\begin{equation} \label{eq:elim-rule}
\hat{\delta}_{ij,t}  + w_{ij,t}=  \hat{\mu}_{i,t} - \hat{\mu}_{j,t} + w_{ij,t}< 0.
\end{equation}
i.e., if the upper confidence of $\delta_{ij}$ is below $0$. To see why (\ref{eq:elim-rule}) achieves statistical validity, notice that if the confidence bands achieve perfect coverage, i.e., $\delta_{ij}(\theta^c) \in [\hat{\delta}_{ij,t} - w_{ij,t}, \hat{\delta}_{ij,t} + w_{ij,t}]$ for all $i,j$ and $t$, then for any design $i \neq b$ (recall that $b$ is the unique best design),
\begin{equation*}
\hat{\mu}_{b,t}  - \hat{\mu}_{i,t} + w_{bi,t} \ge \mu_b(\theta^c) - \mu_i(\theta^c)  > 0,
\end{equation*}
meaning that design $b$ will never be eliminated. Meanwhile, since $w_{ij,t} \rightarrow 0$ as $t\rightarrow \infty$, the procedure terminates in finite time almost surely, and the final output must be design $b$.

In applying the SE procedure, the core question is how to construct the confidence bands $\{w_{ij,t}\}$. Notice that the false selection happens at any stage when the true optimal design is eliminated. Therefore, we can write the probability of false selection (PFS) as:
\begin{equation} \label{eq:SE}
\begin{aligned}
   \text{PFS} =&\mathbb{P}(\text{The optimal design } b \text{ is eliminated at some } t)\\
    =&\mathbb{P}\left( \bigcup _{t=1}^{\infty} \left\{ \text{The optimal design } b \text{ is eliminated at } t\right\}\right)\\
    = & \sum_{t=1}^{\infty}\mathbb{P} \left( \text{The optimal design } b \text{ is eliminated at } t\right)\\
    \le&\sum_{t=1}^{\infty}  \mathbb{P} \left(\bigcup_{i<j}  \{ \mid \hat{\delta}_{ij,t} - \delta_{ij}(\theta^c) \mid > w_{ij,t}\} \right),
\end{aligned}
\end{equation}
 where the second equality follows from the fact that the events \{The optimal design $b$  is eliminated at  $t$\}, $t=1,2,\ldots$ are  disjoint, and the last inequality follows from
$$ \{\text{The optimal design } b \text{ is eliminated at } t \} \subseteq \bigcup_{i<j}  \{ \mid \hat{\delta}_{ij,t} - \delta_{ij}(\theta^c) \mid > w_{ij,t}\}. $$
Hence, if we can find $w_{ij,t}$ such that
$$
    \mathbb{P} \left(\bigcup_{\substack{ i<j} } \{ \mid \hat{\delta}_{ij,t} - \delta_{ij}(\theta^c) \mid > w_{ij,t}\} \right) \le \dfrac{6\alpha}{\pi^2t^2},
$$
then we have
\begin{equation} \label{eq:SE2}
    \begin{aligned}
    &\sum_{t=1}^{\infty}  \mathbb{P} \left(\bigcup_{i<j}  \{ \mid \hat{\delta}_{ij,t} - \delta_{ij}(\theta^c) \mid > w_{ij,t}\} \right)\\
    \le&\sum_{t=1}^{\infty} \dfrac{6\alpha}{\pi^2t^2} = \alpha  
    \end{aligned}
\end{equation}
We obtain that PCS is at least $1-\alpha$.

Notice that we can also use a non-pairwise framework which derives the confidence bands $w_{i,t}$ for each $\mu_i$. Nonetheless, the pairwise framework allows for the use of CRN, which often sharpens the comparison between designs. Furthermore, the common input distribution may introduce additional positive correlation that further reduces variance. All things considered, it is promising to achieve $w_{ij,t} < w_{i,t} + w_{j,t}$ and therefore faster termination. Therefore, we only focus on the pairwise framework.

The problem is reduced to finding tight confidence bands for $\hat{\delta}_{ij,t}$. In the classical IU-free setting where $\hat{\delta}_{ij,t}$  are averages of i.i.d. samples, many concentration bounds (e.g., Hoeffding's inequality, Gaussian tail bounds) are readily available. Unfortunately, these bounds do not apply directly to the case with IU. In the upcoming section, we derive IU-compatible confidence bands, on top of which we build SE procedures.
\section{SEIU Procedures} \label{sec:exact}
We propose procedures based on confidence bands derived through an exact approach. These confidence bands allow us to build idealized SE procedures, which assume full knowledge on several key parameters. The idealized procedures are shown to be statistically valid and are equipped with upper bounds on their expected runtime.  We also briefly discuss how to estimate the unknown parameters in practice to end this section.

\subsection{Derivation of exact confidence bands} \label{sec:SEIU-derivation}
We recall a few notations and concepts before deriving the confidence bands. Denote by $\Theta \subseteq \R^d$ the space in which $\theta$ lives. The stream of transformed input data for the $s$th input distribution, denoted by $\{D_{s,t}\}_t$, are i.i.d. random vectors in $\R^{d_s}$ whose sample mean is an unbiased estimator of $\theta_s^c$. Also recall that a random variable $X$ with mean $\mu=\E[X]$ is called \emph{sub-Gaussian} if there exists $\sigma>0$ such that $\E[e^{s(X-\mu)}] \le \exp(\sigma^2 s^2/2)$ for all $s \in \R$, in which case we write $X \sim subG(\sigma^2)$ and $\sigma^2$ is called a \emph{variance proxy}. In the following assumption, $D_{s,t,j}$ denotes the $j$th component of the random vector $D_{s,t}$.

\begin{assumption} \label{assump:SE-IU}
\quad
\begin{enumerate}
\item[(i)]
$\Theta$ is a compact subset of $\R^d$.
\item[(ii)]
For any $s,t$ and $j$, $D_{s,t,j} \sim subG(\nu_s^2)$ for some $\nu_s>0$ and  $\{D_{s,t,j}\}$ are i.i.d..
\item[(iii)]
For any $\theta \in \Theta$ and $i \in \mathcal{I}$, $X_i(\theta)\sim subG(\bar{\sigma}_i^2)$ for some $\bar{\sigma}_i > 0$ independent of $\theta$.
\item[(iv)]
For any design $i$, the performance function $\mu_i$ is Lipschitz continuous w.r.t. input parameter $\theta$. More specifically, there exists $\bar{L}_i>0$ such that, 
\begin{equation*}
|\mu_i(\theta_1) - \mu_i(\theta_2)| \le \bar{L}_i \|\theta_1 - \theta_2 \|_1, \quad \forall \theta_1, \theta_2 \in \Theta,
\end{equation*}
where $\|\cdot\|_1$ denotes the $\ell^1$-norm.
\end{enumerate}
\end{assumption}

In Assumption \ref{assump:SE-IU}, the major condition is the compactness of $\Theta$, which can be satisfied in most real-world applications where input distributions are supported on a bounded set (e.g., number of visitors on a given day, customer waiting time, etc.). With compactness, $\bar{\sigma}_i^2$ can often be taken as the maximum variance proxy over $\Theta$, and the Lipschitz continuity of $\mu_i$ becomes a local property that usually holds in practice. Also, the boundedness of input data and simulation outputs implies their sub-Gaussianity. Furthermore, Assumption \ref{assump:SE-IU} implies that there exists ${\bar{\sigma}}_{ij} > 0$ such that $X_i(\theta) - X_j(\theta) \sim subG({\bar{\sigma}}^2_{ij})$ for all $\theta \in \Theta$, and there exists $\bar{L}_{ij} > 0$ such that $\delta_{ij}(\cdot) = \mu_i(\cdot) - \mu_j(\cdot)$ is $\bar{L}_{ij}$-Lipschitz continuous on $\Theta$. 

We now motivate exact confidence bands for the moving average estimator $\hat{\delta}_{i,t}$. The goal is to find bounds $B_t(\cdot)$ such that for any fixed $x >0$, we have
\begin{equation*}
\mathbb{P}\{ |\hat{\delta}_{ij,t} - \delta_{ij}(\theta^c)| > x\} \le B_t(x),
\end{equation*}
where $B_t(x) \rightarrow 0$ as $t \rightarrow \infty$. To gain intuition for an IU-compatible bound, consider a simplified setting where $\theta^c \in \R$ and there is no SU, i.e.,
\begin{equation} \label{eq:no-SU}
\hat{\delta}_{ij,t} = \frac{1}{t} \sum_{\ell = 1}^t \delta_{ij}(\hat{\theta}_\ell).
\end{equation}
in which case error only comes from the biases $\{\delta_{ij}(\hat{\theta}_\ell) - \delta{ij}(\theta^c)\}$. Then, for any $x>0$,
\begin{align}
\mathbb{P}\{ |\hat{\delta}_{ij,t} - \delta_{ij}(\theta^c)| > x\} 
&\le \mathbb{P}\left\{\frac{1}{t} \sum_{\ell=1}^t |\delta_{ij}(\hat{\theta}_\ell) - \delta_{ij}(\theta^c)| > x \right\} \le
 \mathbb{P}\left\{ \sum_{\ell=1}^t |\hat{\theta}_\ell - \theta^c| > tx /\bar{L}_{ij} \right\} \label{eq:bound-tmp}
\end{align}
by the Lipschitz continuity of $\delta_{ij}$. To further decompose the last term in (\ref{eq:bound-tmp}), notice that for any nonnegative sequence $\{\omega_\ell\}$ with $\sum_{\ell=1}^t \omega_\ell = 1$, we have
\begin{align*}
\mathbb{P}\left\{ \sum_{\ell=1}^t |\hat{\theta}_\ell - \theta^c| > tx /\bar{L}_{ij} \right\}
\le \mathbb{P}\left\{\bigcup_{\ell=1}^t \left\{|\hat{\theta}_\ell - \theta^c| > \omega_\ell tx /\bar{L}_{ij} \right\} \right\}
\le \sum_{\ell=1}^t \mathbb{P}\left\{|\hat{\theta}_\ell - \theta^c| > \omega_\ell tx /\bar{L}_{ij} \right\}.
\end{align*}
The question is whether we can choose $\{\omega_\ell\}$ such that the resulting bound converges to 0 as $t\rightarrow \infty$. Note that taking equal weights $\omega_\ell = 1/t$ is not helpful, since the resulting bound is bounded from below by a positive constant:
\begin{equation*} 
\sum_{\ell=1}^t \mathbb{P}\left\{|\hat{\theta}_\ell - \theta^c| > x /\bar{L}_{ij} \right\} \ge \mathbb{P}\left\{|\hat{\theta}_1 - \theta^c| > x /\bar{L}_{ij} \right\} =  \mathbb{P}\left\{|D_1 - \theta^c| > x /\bar{L}_{ij} \right\}.
\end{equation*}
Instead, an intuitive way to tighten this bound is to allow larger deviations (assign a greater $\omega_\ell$) for earlier stages, and smaller deviations for later stages. This idea is made concrete by the following proposition.

\begin{proposition} \label{prop:cb}
Let Assumption \ref{assump:SE-IU} hold. Then, for any design $i$ and any $x,y>0$,
\begin{align} \label{eq:single-cb}
\mathbb{P}\left\{|\hat{\delta}_{ij,t} -\delta_{ij}(\theta^c)| > x + y \right\} &\le 2 \exp\left\{- \frac{[M(t) - M(t_\eta)]x^2}{2  \bar{\sigma}_{ij}^2} \right\} \notag \\
&+ 2(t-t_\eta) \Sigma_{s=1}^S d_s \exp\left\{- \frac{[M(t) - M(t_\eta)]^2 y^2}{2d^2 \bar{L}_{ij}^2 \nu_s^2 \gamma^2_{s,\eta}(t)} \right\},
\end{align}
where $\gamma_{s,\eta}(t) := \sum_{\ell = t_\eta+1}^t m(\ell) / \sqrt{N_s(\ell)}$.
\end{proposition}

\textbf{remark}
The term $\gamma_{s,\eta}$ loosely characterizes the rate at which the IU-induced error accumulates over time: the $1/\sqrt{N_s(\ell)}$ term reflects how fast IU (i.e., estimation error in $\hat{\theta}_\ell$) diminishes as input data grows, but the error is multiplied by a factor of $m(\ell)$, since during each stage we perform $m(\ell)$ simulation runs under the inaccurate input distribution $P_{\hat{\theta}_\ell}$.

\textbf{remark}
In the proof of Proposition \ref{prop:cb}, the choice of $\omega_\ell \propto m(\ell) / \sqrt{N_s(\ell)}$ echoes our interpretation of $\gamma_{s,\eta}$, and it reweighs IU at different stages to yield a useful upper bound. It is also worth noting that despite the multiplicative factor $(t-t_\eta)$ in (\ref{eq:single-cb}), the bound still converges to 0 exponentially fast as $t\rightarrow \infty$.

Proposition \ref{prop:cb} provides an IU-compatible bounds for a pair of designs. When it comes to designing SE procedures, however, it is more advantageous to use the following simultaneous bounds for all pairs of designs, which exploits the fact that all designs share the same input distribution, and is therefore tighter than directly applying Bonferroni's inequality to (\ref{eq:single-cb}). 

\begin{lemma} \label{lemma:cb}
Let Assumption \ref{assump:SE-IU} hold. Then, for any $x_{ij},y_{ij}>0$,
\begin{align} \label{eq:multi-cb}
\mathbb{P} \left\{\bigcup_{ i<j } \left\{|\hat{\delta}_{ij,t} - \delta_{ij}(\theta^c)| > x_{ij} + y_{ij} \right\}  \right\} 
& \le \sum_{ i<j } \exp\left\{- \frac{[M(t) - M(t_\eta)]x_{ij}^2}{2  \bar{\sigma_{ij}}^2} \right\} \notag \\
& + 2(t-t_\eta)\Sigma_{s=1}^S d_s \max_{ i<j} \left\{\exp\left\{- \frac{[M(t) - M(t_\eta)]^2 y_{ij}^2}{2 \nu_s^2 d^2 \bar{L}^2_{ij} \gamma^2_{s,\eta}(t)} \right\}\right\},
\end{align}
where $\gamma_{s,\eta}(t) := \sum_{\ell = t_\eta+1}^t m(\ell) / \sqrt{N_s(\ell)}$.
\end{lemma}

\subsection{The SEIU procedures} \label{sec:SEIU}
Based on Lemma \ref{lemma:cb}, we design the SEIU procedure based on exact confidence bands derivation, where ``SEIU'' stands for \emph{Sequential Elimination under Input Uncertainty}. The procedure relies on the key parameters $\nu, \bar{\sigma}_{ij}$, and $\bar{L}_{ij}$ which are assumed to be known for the moment. We defer the discussion on how to estimate these parameters to Section \ref{sec:para-est}.

\textbf{Procedure: SEIU}

\begin{itemize}
\item
{\bf Initialization.}
 $\alpha \in (0,1), \eta \in (0,1), \mathcal{S} =\mathcal{I}, t=1$.

\item
{\bf Step 1.} At stage $t$, update the estimate of $\theta^c$ using new input data and run additional simulations. For each pair of designs $i > j$ in $\mathcal{S}$, compute $\hat{\delta}_{ij,t} = \hat{\mu}_{i,t} - \hat{\mu}_{j,t}$ using (\ref{eq:mae}).

\item
{\bf Step 2.} Compute the confidence bands $w_{ij,t} = u_{ij,t} + v_{ij,t}$ for each pair of designs $i<j$ using 
\begin{align}
u_{ij,t} = 2{\bar{\sigma}}_{ij} \sqrt{\frac{\ln\left(\sqrt{\frac{K(K-1)\pi^2}{3\alpha}}t\right)}{M(t) - M(t_\eta)}}, \label{eq:SEIU-su} 
\end{align}
\begin{align}
v_{ij,t} = \max_{1\le s \le S}\left\{\frac{ d\nu_s \bar{L}_{ij} \gamma_{s,\eta}(t)}{M(t) - M(t_\eta)} \sqrt{2\ln\left({\frac{2Sd_s \pi^2 (1-\eta)t^3}{ 3\alpha}} \right)} \right\}, \label{eq:SEIU-iu} 
\end{align}
where $\gamma_{s,\eta}(t) := \sum_{\ell = t_\eta+1}^t m(\ell) / \sqrt{N_s(\ell)}$.

\item 
{\bf Step 3.} 
For each design $i \in \mathcal{I}$, if 
\begin{equation} \label{eq:elim-cond}
    \min_{i > j} \left\{\hat{\delta}_{ij,t}  + w_{ij, t} \right\} < 0,
\end{equation}
then set $\mathcal{S} \gets \mathcal{S} \setminus \{i\}$. Go to {\bf Output} if $|\mathcal{S}|=1$; otherwise, set $t \gets t+1$ and go to \textbf{Step 1}.

\item
{\bf Output.} 
Output the only design in $\mathcal{S}$.
\end{itemize}

In the SEIU procedure presented above, $u_{i,t}$ and $v_{i,t}$ are chosen to ensure that PCS $\ge 1-\alpha$. Specifically, by substituting $x$ and $y$ with $u$ and $v$ in Lemma \ref{lemma:cb}, we have

\begin{equation*}
 \mathbb{P} \left\{\bigcup_{i<j} \left\{|\hat{\delta}_{ij,t} - \delta_{ij}(\theta^c)| > u_{ij,t} + v_{ij,t} \right\} \right\}  \le \dfrac{ 6 \alpha }{\pi^2 t^2},
\end{equation*} i.e., the probability of any design's estimate at any stage exiting its confidence band does not exceed $\dfrac{ 6 \alpha }{\pi^2 t^2}$, and as a result, the probability of false selection does not exceed $\alpha$, as shown in (\ref{eq:SE}) and (\ref{eq:SE2}). To run SEIU, a user first needs to specify the target PCS level $1-\alpha$ and a drop rate $\eta$ for the moving average estimator $\hat{\mu}_{i,t}$. Then, the procedure collects input data and runs simulations alternately, and removes design(s) from the survival set $\mathcal{S}$ whenever the elimination criterion (\ref{eq:elim-cond}) is met. Finally, the last surviving design is output as the selected best.

The following theorem shows the statistical validity of the procedure SEIU and characterizes the number of stages needed for the procedure to terminate, denoted as $\tilde{\tau}$, assuming full knowledge of $\nu_s, \bar{\sigma}_{ij}$ and $\bar{L}_{ij}$:

\begin{theorem} \label{thm:SEIU-guarantee}
Let Assumption \ref{assump:SE-IU} hold, and suppose that $\nu_s$,  ${\bar{\Sigma}}_{ij}$ and $\bar{L}_{ij}$ are known. Then, SEIU selects the best design with probability at least $1-\alpha$. Furthermore, 

\begin{equation} \label{eq:tau-bound}
\E[\tilde{\tau}] \le \max_{i \neq b} \min_{j \neq i} \tilde{\tau}_{ij} + \alpha,
\end{equation}
where
\begin{equation} \label{eq:tau_ij}
\tilde{\tau}_{ij} =\inf\{t \mid {w}_{ij,t} <\delta_{ij}(\theta^c) \} 
\end{equation}

\end{theorem}

The intuition behind the bound on $\E[\tilde{\tau}]$ is that if all confidence bands achieve perfect coverage, then they will eventually become narrow enough for us to distinguish between any two designs.  Hence, $\tilde{\tau}_{ij}$ is the worst-case time it takes for a design $j$ to eliminate design $i$, which would be infinity if $j$ is inferior to $i$. By taking the minimum of $\tilde{\tau}_{ij}$ over all designs $j$ other than $i$, we obtain the longest survival time for $i$. Thus, the total runtime will not exceed the maximum survival time of all suboptimal designs. Although $\tilde{\tau}_{ij}$ does not have a simple closed form, it clearly depends on two factors: (i) the widths of the confidence bands  $\{{w}_{ij,t}\}$; (ii) the difference between means, i.e., $\delta_{ij}(\theta^c)$. 

Like many elimination-based procedures, SEIU can be pre-stopped at any stage to output a subset $\mathcal{S}$ that contains the true best with probability at least $1-\alpha$. Moreover, the following corollary provides a selection at any stage with an indifference zone (IZ) type guarantee.

\begin{corollary} \label{co:SEIU}
Let Assumption \ref{assump:SE-IU} hold, and suppose that $\nu_s,  \bar{\sigma}_{ij}$, and $\bar{L}_{ij}$ are known. Following the SEIU procedure, at any stage $t$, with the remaining set of designs $\mathcal{S}$, the design with the largest moving average mean, i.e., the design $ i^* = \arg\max_{i \in \mathcal{S}_t}\hat{\mu}_{i,t}$, is $\epsilon_t$-optimal with probability $1-\alpha$. That is, with probability $1-\alpha$, $\mu_{i^*}(\theta^c) \ge \max\limits_{1\le i\le K}\mu_i(\theta^c) - \epsilon_t$, where $\epsilon_t = \max\limits_{j \in \mathcal{S}, j \neq i^*} w_{i^*j,t}$. 
\end{corollary}

Corollary \ref{co:SEIU} implies that before the procedure eliminates all the inferior designs, one can also make the selection at any intermediate stage $t$ by choosing the remaining design that has the largest sample mean. Such a selection will give an $\epsilon_t$-optimal design with confidence level $1-\alpha$. The value of $\epsilon_t$ depends on the confidence bands $w_{ij,t}$, which is
 the sum of $\{u_{ij,t}\}$ and $\{v_{ij,t}\}$ that account for SU and IU, respectively. 
 
 To gain more insight, consider a special case where $m(\ell) \equiv m_0$ and $n_s(\ell) \equiv n_0$, i.e., constant batch size of data and simulation over different stages. Then, (\ref{eq:SEIU-su}) and (\ref{eq:SEIU-iu}) can be simplified as
\begin{align}
u_{ij,t} = 2\frac{\bar{\sigma}_{ij}}{\sqrt{m_0}} \sqrt{\frac{\ln\left(\sqrt{\frac{K(K-1)\pi^2}{3\alpha}}t\right)}{t - t_\eta}}, \label{eq:u_it-equal}
\end{align}
\begin{align}
v_{ij,t} \approx  \max_{1\le s \le S} \left\{\frac{ 2d\nu_s \bar{L}_i}{\sqrt{n_0} (1+ \sqrt{\eta})} \sqrt{\frac{2\ln\left(\frac{2S d_s \pi^2 (t-t_\eta)^3}{3(1-\eta)^2 \alpha} \right)}{t}} \right\} 
. \label{eq:v_it-equal}
\end{align}
Thus, $ u_{i,t}$  and  $v_{i,t}$ are both of order $O(\sqrt{\ln t /t})$, which is common in SE procedures. Moreover, it can be seen from (\ref{eq:u_it-equal}) that the width of $u_{i,t}$ primarily depends on the magnitude of $\bar{\sigma}_{ij}/\sqrt{m_0}$, i.e., the discounted variance proxy after averaging $m_0$ output samples at each stage; clearly, a smaller $\bar{\sigma}_{ij}$ and a higher $m_0$ yield narrower bands. The drop rate $\eta$ also matters: notice that $u_{ij,t} \rightarrow \infty$ as $\eta \rightarrow 1$, meaning that the confidence band needs to expand in order to cover inflated SU with more samples thrown away.

Also clear from (\ref{eq:v_it-equal}) is that greater values of $\nu_s / \sqrt{n_0}$ and $\bar{L}_{ij}$, which capture the extent of IU and $\delta_{ij}$'s sensitivity to IU, respectively, would result in a wider $v_{ij,t}$. It can further be checked that $v_{ij,t} \rightarrow 0$ as $\eta \rightarrow 1$, which agrees with the intuition that the more stages we discard, the less we suffer from the biases $\{\delta_{ij}(\hat{\theta}_t) - \delta_{ij}(\theta^c)\}$ from previous stages.

In the SEIU procedure, we require full knowledge of the parameters $\bar{\sigma}_{i,j}, \nu_s, \bar{L}_{i,j}$, which is often unknown  and needs to be estimated in practice. We propose some heuristic methods and discussions for estimating thees unknown parameters, which can be found in the Electronic Companion \ref{ec:para-est}.

\section{ SEIU-MCB Procedures} \label{sec:SEIU-MCB}

In this section, we propose a more efficient SEIU-MCB procedure by deriving asymptotically valid confidence bands, which utilizes the result of Multiple Comparisons with the Best (MCB) proposed by \cite{chang1992optimal}. MCB is extended by \cite{song2019input} to construct confidence bands accounting for IU, where they exploit the jointly asymptotic normality of pairwise difference of performance estimators (referring to $\{\hat{\delta}_{ij,t}\}_{j\neq i}$) that consists of samples under the current estimated input distribution. We extend this jointly asymptotic normality to our moving-average estimator. The major challenge of establishing the asymptotic normality result for our moving-average estimator is that the estimator aggregates simulation outputs from different stages, which disallows the application of the central limit theorem by a direct partition of the total error into two parts caused respectively by IU and SU. Instead we adopt the Lindeberg-Feller Theorem (see proposition 2.2.7 in \cite{van2000asymptotic} ) to prove the asymptotic normality of our moving-average estimator.

Compared with SEIU, the SEIU-MCB procedure requires much less restrictive assumptions and achieves higher efficiency. Furthermore, the MCB framework helps to  avoid the usage of Boole's inequality across designs, and as a result, the asymptotic normality result yields much tighter confidence bands to control the cumulative error across stages. In addition,  our result explicitly characterizes how $\eta$ affects the tradeoff between IU and SU, and shows that the elimination  between any two designs can be potentially boosted by optimizing $\eta$.

\subsection{Multiple Comparison with the Best (MCB)} \label{sec:MCB}

The MCB result is shown as the following theorem.

\begin{theorem} \label{thm:MCB}
Let $\hat{\mu}_i$ be the estimator of $\mu_i$ for $i=1,2,\cdots,K$, $x^+ = max(x,0)$ and $x^- = -min(x,0)$. If for each fixed $i$,
\begin{equation} \label{eq:pairwise CI}
    P\left\{  \widehat{\mu}_{i,t} -\mu_i -  (\widehat{\mu}_{j,t} -\mu_j) \ge -w_{ij,t} \textbf{ for }i \neq j\right\} =1-\alpha,  
\end{equation}  
then we can make the joint probability statement
\[ \mathbb{P}\left\{ \mu_i - \max_{j\neq i} \mu_j \in [W_i^-,W_i^+] \quad \textbf{for i }  = 1,2,\cdots K 
\right\} = 1-\alpha,\]
where $W_i^-$ and $W_i^+$ can be computed as 
\begin{equation} \label{eq:D}
\begin{aligned}
     &W_i^+ = \left(\min_{j\neq i} [\widehat{\mu}_{i,t} - \widehat{\mu}_{j,t} + w_{ij,t}]\right)^+, \mathbb{G} = \{i : W_i^+ > 0\}\\ 
      &W_i^- = \left\{ 
      \begin{aligned}
      &0 & \text{if } \mathbb{G} = \{i\} \\
      -& \left(\min_{j\neq i} [\widehat{\mu}_{i,t} - \widehat{\mu}_{j,t} - w_{ji,t}]\right)^-  &\text{otherwise.}
      \end{aligned}\right.
\end{aligned}
\end{equation}
\end{theorem}
Theorem \ref{thm:MCB} says that if we are provided with the simultaneous pairwise confidence bands $w_{ij,t}$ for each $i$, then we can get a simultaneous confidence band for $i$ which measures how the design $i$ is compared with the remaining best. If the upper bound $W_i^+$ is zero, then we eliminate design $i$ since it is inferior to the remaining best design. If $W_i^-$ is zero, then we  directly choose design $i$ as the optimal design. 

Now, the key to the problem remains as how to compute the simultaneous pairwise confidence band $w_{ij,t}$, which depends on the joint distribution of 
$$\Delta_{i,t} \triangleq (\hat{\delta}_{i1,t} - \delta_{i1}(\theta^c),\cdots,\hat{\delta}_{ii-1,t} - \delta_{ii-1}(\theta^c),\hat{\delta}_{ii+1,t} - \delta_{ii+1}(\theta^c),\cdots,\hat{\delta}_{iK,t} - \delta_{iK}(\theta^c)).$$
However, usually the joint distribution is hardly computable, and even so, the computation can be expensive. Hence, instead of finding the exact joint distribution, we show the asymptotic joint distribution in the following Section \ref{sec:normality}.

\subsection{Asymptotic normality} \label{sec:normality}
In Section \ref{sec:exact}, a critical assumption underpinning the SEIU framework is the compactness of input distribution's parameter space $\Theta$, which plays an important role in rigorously bounding the variation of simulation output in the presence of IU. We now relax this condition and shift to the following assumption.
\begin{assumption} \label{assump:CLT}
For all $i \in \mathcal{I}$ and $s \in \{1,2,\ldots, S\}$,
\quad
\begin{enumerate}
\item[(i)]
$\Sigma_{D,s} := Cov(D_{s,1})$ exists.

\item[(ii)]
$\Sigma(\theta)$ exists and is continuous for all $\theta \in \Theta$, where $\Sigma_{ij}(\theta) = Cov[X_i(\theta),X_j(\theta)]$.
\item[(iii)]
$\mu_i(\cdot)$ is twice continuously differentiable in $\Theta$.
\item[(iv)]
$\{n_s(t)\}$ and $\{m(t)\}$ are uniformly bounded. Furthermore, there exist positive constants $\bar{n}_s$ and $\bar{m}$ such that $N_s(t)/t \rightarrow \bar{n}_s$ and $M(t) / t \rightarrow \bar{m}$ as $t\rightarrow \infty$.
\end{enumerate}
\end{assumption}

In Assumption \ref{assump:CLT} (i) and (ii), the existence of input data's covariance and simulation output's variance is far less stringent than the sub-Gaussianity conditions in Assumption \ref{assump:SE-IU} and holds in most real-world applications. The smoothness of $\mu_i(\cdot)$ in (iii) is also a reasonable assumption for many parametric families of $P_\theta$. Moreover, (iv) only requires the limit of input data and simulation batch sizes to exist in the Ces\`{a}ro sense, and the uniform boundedness of $\{n_s(t)\}$ and $\{m(t)\}$ is guaranteed in practice. With Assumption \ref{assump:CLT}, we establish the following asymptotic result for ${\Delta}_{i,t}$.

\begin{theorem}
\label{thm:joint normality}
Let Assumption \ref{assump:CLT} holds. Then, for any $\eta \in [0,1)$ and any design $i$, 
\begin{equation} \label{eq:mae-CLT}
\sqrt{t} \Delta_{i,t} \Rightarrow \mathcal{N}(0, \Sigma_{i, \infty}), \quad \text{ as } t\rightarrow \infty,
\end{equation}
where $\Rightarrow$ means convergence in distribution, $\mathcal{N}$ denotes the normal distribution, and
\begin{equation} \label{eq:limit-variance}
\Sigma_{i, \infty}(j,j^\prime) := \lambda_{I,\eta} \nabla \delta_{ij}(\theta^c)^\intercal \bar{\Sigma}_D \nabla \delta_{ij^\prime} (\theta^c) + \lambda_{S,\eta} \bar{m}^{-1} \textbf{Cov}\left( X_i(\theta^c)-X_j( \theta^c),X_i(\theta^c) - X_{j^\prime}(\theta^c)\right),
\end{equation}
in which $\bar{\Sigma}_D := diag(\bar{n}^{-1}_1 \Sigma_{D,1}, \ldots, \bar{n}^{-1}_S \Sigma_{D,S})$ and 
\begin{equation}  \label{eq:rho}
\lambda_{I,\eta} := \frac{2}{1-\eta} + \frac{2\eta \ln\eta}{(1-\eta)^2}, \quad \lambda_{S,\eta} := \frac{1}{1-\eta}.
\end{equation}

\end{theorem} 

Theorem \ref{thm:joint normality} shows that the asymptotic covariance matrix of $\Delta_{i,t}$ is the weighted sum of two components, 
\begin{equation} \label{eq:sigma-decomp}
\Sigma_{i,IU}(j,j^\prime):= \nabla \delta_{ij}(\theta^c)^\intercal \bar{\Sigma}_D \nabla \delta_{ij^\prime} (\theta^c),  \quad \Sigma_{i,SU}(j,j^\prime):=  \bar{m}^{-1} \textbf{Cov}\left( X_i(\theta^c)-X_j( \theta^c),X_i(\theta^c) - X_{j^\prime}(\theta^c)\right), 
\end{equation}
which quantify the variance induced by IU and SU, respectively. Furthermore, if we let $ \partial_{\theta_s} \delta_{ij}(\theta^c)$ denote the partial derivative $\partial \delta_{ij}(\theta^c) / \partial \theta_s$, then we can decompose $\Sigma_{i,IU}$ as
\begin{equation*}
\Sigma_{i,IU} (j,j^\prime)= \sum_{s=1}^S \partial_{\theta_s} \delta_{ij}(\theta^c)^\intercal \Sigma_{D,s}\partial_{\theta_s} \delta_{ij^\prime}(\theta^c) / \bar{n}_s
\end{equation*}
which attributes the variance to each individual input distribution. The $O(1/\bar{n}_s)$ and $O(1/\bar{m})$ convergence rate in (\ref{eq:limit-variance}) are standard. However, the weights $\lambda_{I,\eta}$ and $\lambda_{S,\eta}$ are less common and therefore deserve a closer look. 

The expressions of $\lambda_{I,\eta}$ and $\lambda_{S,\eta}$ are shaped by two major factors: (i) the structure imposed on $\hat{\theta}_t$ (see Assumption \ref{assump:estimator}); (iii) the choice of estimator $\hat{\mu}_{i,t}$. In particular, the sample average structure of $\hat{\theta}_t$ makes asymptotic analysis more tractable, since it allows the correlated biases $\{\delta_{ij}(\hat{\theta}_t) - \delta_{ij}(\theta^c)\}$ to be decoupled into a weighted sum of i.i.d. samples $D_{s,t}$. To get a sense of the limiting behavior of $\lambda_{I,\eta}$ and $\lambda_{S,\eta}$, we investigate two extreme cases.

\begin{enumerate}
\item[{\bf Case 1}.]
If $\eta \rightarrow 1$, then $\lambda_{I,\eta} \rightarrow 1$ and $\lambda_{S,\eta} \rightarrow  \infty$, which coincides with the intuition that IU can be reduced by dropping samples, but there will be fewer simulation outputs to average out SU.

\item[{\bf Case 2}.]
If $\eta = 0$, then $\lambda_{I,\eta} = 2$ and $\lambda_{S,\eta} = 1$, i.e.,

\begin{equation} \label{eq:0-eta-dynamic}
\sqrt{t} \left[\hat{\delta}_{ij,t} - \delta_{ij}(\theta^c)\right] \Rightarrow \mathcal{N}(0, 2\Sigma_{i, IU} + \Sigma_{i, SU}), \quad \text{as } t\rightarrow \infty.
\end{equation}

To put (\ref{eq:0-eta-dynamic}) in perspective, we cite an asymptotic normality result from \cite{wu2017OCBAIU} for the case of static input data. Suppose that there is only a single batch of input data, which we use to compute the parameter estimate $\hat{\theta}_t$ and then run simulations under $P_{\hat{\theta}_t}$. Let $\tilde{\mu}_{i,t} = \frac{1}{M_i(t)} \sum_{r=1}^{M_i(t)} X_{i,r} (\hat{\theta}_t)$ be the estimator for $\mu_i(\theta^c)$. Under some mild conditions, we have
\begin{equation} \label{eq:0-eta-static}
\sqrt{t} \left[\tilde{\delta}_{ij,t} - \delta_{ij}(\theta^c)\right] \Rightarrow \mathcal{N}(0, \Sigma_{i, IU} + \Sigma_{i, SU}), \quad \text{as } t\rightarrow \infty.
\end{equation}
By comparing (\ref{eq:0-eta-dynamic}) with (\ref{eq:0-eta-static}), it can be seen that the weight of $\Sigma_{i, IU}$ is doubled in the case of streaming input data while the weight of $\Sigma_{i, SU}$ remains the same, where the inflation of $\Sigma_{i, IU}$ is due to estimation error accumulated over multiple stages.\\
\end{enumerate}

With Theorem \ref{thm:joint normality}, we approximate the joint distribution of
 $\Delta_{i,t}$ as:
$$\Delta_{i,t} \mathop{\approx}\limits^D \mathcal{N} (0,\Sigma_{i,\infty} / t). $$
Then $\{w_{ij,t},\ j\neq i\}$ is a $(K-1)$-multidimensional quantile. When given the confidence level $1-\alpha$, the choice of $K-1$-dimensional quantile is not unique since the degree of freedom is $K-1$. We can reduce the degree of freedom to 1 by letting   
$$ \mathcal{P} \left( \hat{\delta}_{ij,t}  - \delta_{ij}(\theta^c) \ge - w_{ij,t} \right)  = \mathcal{P} \left( \hat{\delta}_{ij',t}  - \delta_{ij'}(\theta^c) \ge - w_{ij',t} \right) \quad \forall j \neq j'.$$
i.e., all the confidence bands result in the same probability coverage. This can be done easily under the asymptotic normality of $\hat{\delta}_{ij,t}$. To be specific, for each $i$, we find the smallest $w_{i,t}$, such that 
$$ \mathcal{P} \left( \hat{\delta}_{ij,t}  - \delta_{ij}(\theta^c) \ge - w_{i,t} \sigma_{ij,t},\quad \forall j \neq i\right) \ge 1-\alpha, $$ 
where $\sigma_{ij,t}$ refers to the variance of $\hat{\delta}_{ij,t}$ and we assume $\hat{\delta}_{ij,t}$ approximately follows a normal distribution centered at $\delta_{ij}(\theta^c)$. We present the SEIU-MCB procedure in the next section.

\subsection{SEIU-MCB procedure}

With the elimination and selection rule stated in Section \ref{sec:MCB}, we have the SEIU-MCB procedure as follows:

\paragraph{ \textbf{ SEIU-MCB}}
\begin{itemize}
    \item \textbf{Initialization.} $\alpha \in (0,1), \eta \in (0,1), t=1, \mathcal{S} = \mathcal{I}$ .
    \item \textbf{Step 1.} At stage t, update the estimate of $\theta^c$ using new input data and run additional simulations using CRN. For each system $i \in \mathcal{S}$, compute $\hat{\mu}_{i,t}$ using (\ref{eq:mae}).
    \item \textbf{Step 2.} Compute the simultaneous pairwise confidence bands $w_{ij,t}$ using (\ref{eq:pairwise CI}) and (\ref{eq:mae-CLT}) with confidence level set as $\frac{6\alpha}{\pi^2 t^2}$. Then compute $W_i^-,W_i^+$  using (\ref{eq:D}).
    \item \textbf{Step 3.} For each $i$, if $W_i^- = 0$, then $\mathcal{S} = \{i\}$; esle if $W_i^+ = 0$, $\mathcal{S} \leftarrow \mathcal{S}$ \textbackslash $\{i\}$. If $|\mathcal{S}| = 1$, then go to \textbf{Output}.  Increment $t$ by 1 and go to \textbf{Step 1}.
    \item \textbf{Output.} Output the only design in $\mathcal{S}$ as the optimal system.
\end{itemize}
As we can see from the procedure, compared with the SEIU where we have the sub-Gaussian and Lipschitz continuous assumption, SEIU-MCB requires much less knowledge of those parameter estimates. This is attributed to the general asymptotic normality result established in Section \ref{sec:normality}, by which we can compute the confidence bands easily.

Moreover, from Section \ref{sec:normality}, we can see the drop rate balances the impact of input uncertainty and simulation uncertainty characterized by the asymptotic corvariance matrix. Detailed discussions on minimizing such impact through $\eta$ are presented in \ref{ec:opt-eta}.

\section{Extension to the Indifference Zone (IZ) setting} \label{sec:indifference zone}
In Section \ref{sec:exact} and \ref{sec:SEIU-MCB}, the SEIU and SEIU-MCB procedures are designed to proceed until there is only one design left in the remaining candidate set, because of our goal to identify the unique optimal design. However, in practice, selecting one of the ``near-optimal" designs is often sufficient, which is the motivation behind the indifference-zone R\&S procedures. To be more specific, an IZ R\&S procedure aims to find an $\epsilon$-optimal design with a given $\epsilon > 0$, i.e., a design $i^*$ such that $\mu_{i^*} \ge \max\limits_{1\le i\le K}\mu_i - \epsilon$, where the $\epsilon$ is often referred to as the IZ preference. Such an $\epsilon$ compromise can sometimes dramatically save simulation effort, especially when the optimal performance $\mu_b(\theta^c)$ is close to the second optimal performance $\mu_{b'}(\theta^c)$. Inspired by Corollary \ref{co:SEIU}, we can also extend our procedures to the IZ setting, and terminate the procedure when either of the cases happens: (i) there is only one design remaining, or (ii) $\epsilon_t < \epsilon$.  Denote by SEIU(IZ) and SEIU-MCB(IZ), the procedures with IZ setting. We add an additional step 4 in both procedures to verify whether an $\epsilon$-optimal is found.

\textbf{Step 4 in SEIU(IZ)}

\begin{itemize}

\item Let $i^* = \arg\max\limits_{i \in \mathcal{S}_t} \hat{\mu}_{i,t}$. If 
\[ \max\limits_{j \in \mathcal{S}, j\neq i^*} w_{i^*j,t} < \epsilon,\]
then set $\mathcal{S} \leftarrow \{i^*\}$ and go to {\bf Output}; otherwise, set $t \leftarrow t+1$ and go to \textbf{Step 1}.

\end{itemize}

 \textbf{ Step 4 in SEIU-MCB(IZ)}
\begin{itemize}
    
    \item 
    Let $\mathcal{I} = \left\{ i \in \mathcal{S} | -W_i^- < \epsilon \right\} $. If $|\mathcal{I}| \ge 1$, let $\hat{i}^* = \arg\max_{i \in \mathcal{I}} \hat{\mu}_{i,t}$, set $\mathcal{S} \leftarrow \{\hat{i}^*\}$, and go to \textbf{Output}. Otherwise,  increment $t$ by 1 and go to \textbf{Step 1}.
    
\end{itemize}

With similar intuition as Theorem \ref{thm:SEIU-guarantee}, the following corollary provides the statistical validity of SEIU(IZ) and characterization of the number of stages needed to  terminate the algorithm.

\begin{corollary}\label{co:IZ}
Let Assumption \ref{assump:SE-IU} hold, and suppose that $\nu_s,  \bar{\sigma}_{ij}$, and $\bar{L}_{ij}$ are known.
 Then, SEIU(IZ) selects an $\epsilon$-optimal design with probability at least $1-\alpha$. Furthermore, let $\tau $ denote the number of stages it takes to terminate. Then
\begin{equation} \label{eq:tau}
{\tau} \le \inf\{t \mid \max_{1 \le i < j \le K} {w}_{ij,t} < \epsilon\}  \qquad a.s. 
\end{equation}
\end{corollary}

Corollary \ref{co:IZ} shows that an advantage of using the IZ criterion is that we can characterize the number of stages needed to terminate in a stronger almost sure sense. The intuition is that with the IZ criterion, when the confidence band is small enough we will always either identify the best design or reach specified the accuracy($\epsilon$-optimal).   Again, although we do not have a closed form of the upper bound on $\tau$, it clearly depends on the width of the confidence bands, the true difference $\delta_{ij}(\theta^c)$, and the IZ preference $\epsilon$. 

Another advantage of our IZ procdures, compared with the well-known KN procedure (see \cite{kim2001fully}), is that our procedures compute the confidence bands independently of IZ preference and terminate when either the best selection is made or the IZ preference is satisfied. In contrast, the KN procedure computes the confidence bands by considering a so-called ``slippage configuration", where the best design has expected performance exactly $\delta$ better than the second best. If the actual difference is bigger than $\delta$, the KN procedure will compute a larger confidence band and make redundant simulations. In our procedures, if we set $\delta$ smaller than the true difference of the expected performance between the best and the second best designs, it will not affect the performance of the procedures since they will act at least as good as the SEIU and SEIU-MCB procedures.

\section{Numerical Results} \label{sec:numerical}
 The numerical study mainly consists of three parts: (i) comparing our procedures with the KN++ procedure which does not consider input uncertainty; (ii) investigating the performance of all our proposed procedures; (iii) optimizing $\eta$ to boost the efficiency of SEIU-MCB.

\subsection{Test problems} \label{sec:test-probs}
Two problems are used for numerical testing. One is to minimize a quadratic objective function with a single source of IU, and the other is a more complex production-inventory optimization problem from \cite{hong2009estimating} with multiple independent sources of IU. A detailed description of the two problems can be found in \ref{ec:quad-prob} and \ref{ec:prod-inv-prob}. In all subsequent experiments, we consider the following settings unless otherwise stated.
\begin{itemize}
\item[(i)]
Equal batch size: $n_s(t) = n,  m(t) = m$ for some fixed $n,m  \in \Z^+$.

\item[(ii)]
Random batch size: $n_s(t)$ and $m(t)$ are i.i.d. samples drawn from the uniform distribution on the lattice $\{\bar{k},2\bar{k},\ldots, 5\bar{k}\}$ for some fixed $\bar{k}$.
\end{itemize}
More specifically, in the case of unequal batch size, the random samples of $n_s(t)$ and $m(t)$ are independent of input data and simulation outputs.

\subsection{Experiment results} \label{sec:exp-result}
In the following sections, we first use the simple quadratic example to show the necessity of accounting for input uncertainty by comparing the KN++ procedure with our two IZ-type procedures: SEIU(IZ) and SEIU-MCB(IZ).  We then demonstrate the applicability and effectiveness of all our proposed procedures on the more general production-inventory example, where we have multiple input distributions and allow different batch sizes for both input data and simulation budget. At last, we show  that the performance of SEIU-MCB can be further boosted by optimizing $\eta$. 

\subsubsection{Necessity of considering IU} \label{sec:IU control}
We first use the quadratic problem with equal batch size and $\mathcal{I} = \{ \theta^c\ + 0.3\cdot i: i\in [-5,5] \cap \mathcal{Z}\},$  to test the three procedures with IZ setting:
 SEIU(IZ), SEIU-MCB(IZ), and KN++. The KN++ procedure (\cite{kim2006asymptotic}) is an extension of the KN procedure which allows non-normal observation data and updating the variance estimator. The KN++ procedure uses average batch mean to estimate the variance and updates the variance estimator as more simulation outputs are collected, and hence it can be directly applied to our example with equal batch size. Specifically, it is applied by ignoring the existence of input uncertainty and carrying out simulations in each stage conditioned on the current estimated input parameter. Table  \ref{tab:different delta} shows the empirical PCS and average termination stage along with their 95\% confidence intervals, based on 100 independent runs  of the three procedures with different values of the IZ parameter $\delta$ and different input data batch sizes. The empirical PCS is computed as the percentage of algorithm runs that correctly select the true optimal design. The target PCS is set as $90\%$.  Some observations can be made as follows.

\begin{table}[ht]
\caption{PCS and termination stage for quadratic example (equal batch size)}
\begin{center}
\scalebox{0.8}{
\begin{tabular}{|c||cc|cc|cc}
    \hline\hline
    \multicolumn{7}{c}{$\eta = 0.5,m_s = n_s = 10 $}\\
    \hline
     $\delta$& \multicolumn{2}{c|}{SEIU(IZ)} & \multicolumn{2}{c|}{SEIU-MCB(IZ)} & \multicolumn{2}{c|}{KN++(IZ)}\\
   \hline
   \multirow{2}{*}{0.3} &  \multicolumn{2}{c|}{$98\% \pm 1.7\%$}  &  \multicolumn{2}{c|}{$82\% \pm 5\%$} &  \multicolumn{2}{c|}{$75\% \pm 8\%$}  \\
   &  \multicolumn{2}{c|}{$77 \pm 6.2$}  &  \multicolumn{2}{c|}{$17 \pm 1.9$} &  \multicolumn{2}{c|}{$15 \pm 1.5$} \\
   \hline
   \multirow{2}{*}{0.2} &  \multicolumn{2}{c|}{$99\% \pm 0.6\%$}  &  \multicolumn{2}{c|}{$88\% \pm 4\%$} &  \multicolumn{2}{c|}{$83\% \pm 7\%$}  \\
   &  \multicolumn{2}{c|}{$111 \pm 6.7$}  &  \multicolumn{2}{c|}{$18 \pm 1.4$} &  \multicolumn{2}{c|}{$21 \pm 1.3$} \\
   \hline
   \multirow{2}{*}{0.1} &  \multicolumn{2}{c|}{$100\% \pm 0\%$}  &  \multicolumn{2}{c|}{$90\% \pm 2\%$} &  \multicolumn{2}{c|}{$85\% \pm 6\%$}  \\
   &  \multicolumn{2}{c|}{$521 \pm 31$}  &  \multicolumn{2}{c|}{$25 \pm 0.9$} &  \multicolumn{2}{c|}{$34 \pm 5.4$} \\
   \hline\hline
  \end{tabular}} \qquad
  \scalebox{0.8}{ \begin{tabular}{|c||cc|cc|cc}
   \hline \hline
     \multicolumn{7}{c}{$\eta = 0.5,m_s = 10, \delta = 0.2$}\\
    \hline
     \ $n_s$& \multicolumn{2}{c|}{SEIU(IZ)} & \multicolumn{2}{c|}{SEIU-MCB(IZ)} & \multicolumn{2}{c|}{KN++(IZ)}\\
   \hline
   \multirow{2}{*}{5} &  \multicolumn{2}{c|}{$99\% \pm 0.6\%$}  &  \multicolumn{2}{c|}{$87\% \pm 5\%$} &  \multicolumn{2}{c|}{$74\% \pm 5\%$}  \\
   &  \multicolumn{2}{c|}{$170 \pm 9.7$}  &  \multicolumn{2}{c|}{$25 \pm 1.6$} &  \multicolumn{2}{c|}{$24 \pm 1.5$} \\
   \hline
   \multirow{2}{*}{10} &  \multicolumn{2}{c|}{$99\% \pm 0.6\%$}  &  \multicolumn{2}{c|}{$88\% \pm 4\%$} &  \multicolumn{2}{c|}{$83\% \pm 7\%$}  \\
   &  \multicolumn{2}{c|}{$111 \pm 6.7$}  &  \multicolumn{2}{c|}{$18 \pm 1.4$} &  \multicolumn{2}{c|}{$21 \pm 1.3$} \\
   \hline
   \multirow{2}{*}{15} &  \multicolumn{2}{c|}{$97\% \pm 1.8\%$}  &  \multicolumn{2}{c|}{$89\% \pm 4\%$} &  \multicolumn{2}{c|}{$86\% \pm 3\%$}  \\
   &  \multicolumn{2}{c|}{$89 \pm 5.0$}  &  \multicolumn{2}{c|}{$16 \pm 0.9$} &  \multicolumn{2}{c|}{$18 \pm 1.0$} \\
   \hline\hline
    \end{tabular}}
\end{center}
 \label{tab:different delta}
 \end{table}

\begin{enumerate}
\item
The empirical PCS of the KN++ procedure is less than the target PCS 0.9 across all settings.  The true difference between the best design and the second best is $0.16$. Hence, with $\delta = 0.1$, it is sufficient to identify the best design. Furthermore, by examining the estimated input parameters when incorrect selection or elimination happens, we found that the KN++ procedure makes the false selection mainly when procedures terminate in earlier stages, where input parameters are estimated more inaccurately. This implies the necessity of controlling the input error. 
\item
 The SEIU-MCB(IZ) achieves the target PCS when $\delta = 0.1$ and shoots a higher PCS than KN++ in all settings with mostly earlier termination. This implies the high efficiency of the SEIU-MCB procedure, even when it takes IU into consideration.
\item Comparing SEIU(IZ) and SEIU-MCB(IZ), the latter procedure requires much less simulation effort and achieves more accurate empirical PCS.  It shows that the usage of asymptotic normality and MCB in SEIU-MCB significantly relieves the conservativeness in the SEIU procedure. The  empirical  PCS  of SEIU(IZ) always overshoots the target PCS, revealing the conservativeness of the procedure. The conservativeness is mainly due to the Bonferroni inequality used in deriving the confidence bands in the SEIU procedure.
\end{enumerate}

\subsubsection{Inventory-production problem: a more general setting} \label{sec:Inv-Pro example}
 In this section, we show the applicability and effectiveness of all our proposed procedures using an inventory-production example, where we have multiple unknown input distributions, with random batch sizes for both input and simulation data. We set the maximal production amount $R^* = 4$, holding  cost $c_H = 0.5$, and backlog cost $c_B = 1$. We consider two scenarios with different input dimension $S$ and input parameter $\theta^c$: $S = 2$, $\theta^c = [4, 5]^\intercal$; and $S=4$, $\theta^c = [4, 5, 3, 3]^\intercal$.  The optimal order-up-to quantity is selected among $\mathcal{I} =  \{1, 2, \ldots, 10\}$. For IZ procedures, we set the IZ preference $\delta = 0.3$.  Table \ref{tab:inventory} shows the empirical PCS and average termination stage along with their 95\% confidence intervals, based on 100 independent runs of each procedure. We summarize our observations as follows:
\begin{enumerate}
    \item The SEIU and SEIU(IZ) procedures obtain the $100\%$ empirical PCS across all settings, overshooting the target PCS. The improved procedures SEIU-MCB and SEIU-MCB(IZ) significantly relieve the conservativeness and achieve an empirical PCS much closer to the target while terminating in much less number of stages.
    \item Comparing the two scenarios where the number of unknown input distributions $S=2$ and $S=4$, the actual difference between the best design and the second best is similar, but all four procedures tend to terminate in later stages when $S= 4$ than $S = 2$. Specifically, the termination stages of SEIU and SEIU(IZ) increase roughly by a factor of 3, while those of SEIU-MCB and SEIU-MCB(IZ) increase by a factor of less than 2. Their empirical PCS also increase and  overshoot the target PCS slightly when $S=4$. This implies higher dimension of input uncertainty can cause a larger impact on both the PCS and the termination stage, and SEIU-MCB and SEIU-MCB(IZ) are more robust to this impact.
    
\end{enumerate}
\begin{table}[ht]
   \caption{PCS and termination stage for production inventory problem with random batch size ($\eta = 0.5 $)}
  \vspace{-5mm}
    \begin{center}  
    \scalebox{0.8}{
   \begin{tabular}{|c||cc|cc|cc|cc|}
   \hline \hline
    
    \multicolumn{9}{c}{$\theta^c = [4, 5]^\intercal (S=2)$ }\\
    \hline
     \ $\bar{k}$& \multicolumn{2}{c|}{SEIU} & \multicolumn{2}{c|}{SEIU(IZ)} & \multicolumn{2}{c|}{SEIU-MCB}&
     \multicolumn{2}{c|}{SEIU-MCB(IZ)}\\
   \hline
   \multirow{2}{*}{10} &  \multicolumn{2}{c|}{$100\% \pm 0\%$}  &  \multicolumn{2}{c|}{$100\% \pm 0\%$} &  \multicolumn{2}{c|}{$90\% \pm 4\%$} &  \multicolumn{2}{c|}{$81\% \pm 6\%$}  \\
   &  \multicolumn{2}{c|}{$364 \pm 16$}  &  \multicolumn{2}{c|}{$205 \pm 13$} &  \multicolumn{2}{c|}{$44 \pm 4.7$} 
   & \multicolumn{2}{c|}{$17 \pm 1.2$}\\
   \hline
   \multirow{2}{*}{20} &  \multicolumn{2}{c|}{$100\% \pm 0\%$}  &  \multicolumn{2}{c|}{$100\% \pm 4\%$} &  \multicolumn{2}{c|}{$94\% \pm 3\%$} &  \multicolumn{2}{c|}{$87\% \pm 5\%$} \\
   &  \multicolumn{2}{c|}{$156 \pm7.9$}  &  \multicolumn{2}{c|}{$90 \pm 5.5$} &  \multicolumn{2}{c|}{$25 \pm 2.4$}&  \multicolumn{2}{c|}{$11 \pm 0.6$} \\
   \hline
   \multirow{2}{*}{30} &  \multicolumn{2}{c|}{$100\% \pm 0\%$}  &  \multicolumn{2}{c|}{$100\% \pm 4\%$} &  \multicolumn{2}{c|}{$96\% \pm 2\%$}  &  \multicolumn{2}{c|}{$90\% \pm 4\%$}\\
   &  \multicolumn{2}{c|}{$102 \pm 5.7$}  &  \multicolumn{2}{c|}{$54 \pm 3.8$} &  \multicolumn{2}{c|}{$19 \pm 1.4$}&  \multicolumn{2}{c|}{$7.5 \pm 0.4$} \\
   \hline\hline
   
    \end{tabular}
    \quad\quad
    \begin{tabular}{|c||cc|cc|cc|cc|}
   \hline \hline
    \multicolumn{9}{c}{$\theta^c = [4, 5, 3, 3]^\intercal( S=4)$ }\\
    \hline
     \ $\bar{k}$& \multicolumn{2}{c|}{SEIU} & \multicolumn{2}{c|}{SEIU(IZ)} & \multicolumn{2}{c|}{SEIU-MCB}&
     \multicolumn{2}{c|}{SEIU-MCB(IZ)}\\
   \hline
   \multirow{2}{*}{10} &  \multicolumn{2}{c|}{$100\% \pm 0\%$}  &  \multicolumn{2}{c|}{$100\% \pm 0\%$} &  \multicolumn{2}{c|}{$93\% \pm 3\%$} &  \multicolumn{2}{c|}{$90\% \pm 4\%$}  \\
   &  \multicolumn{2}{c|}{$1137 \pm 152$}  &  \multicolumn{2}{c|}{$528 \pm 84$} &  \multicolumn{2}{c|}{$60 \pm 9.6$} 
   & \multicolumn{2}{c|}{$30 \pm 3.1$}\\
   \hline
   \multirow{2}{*}{20} &  \multicolumn{2}{c|}{$100\% \pm 0\%$}  &  \multicolumn{2}{c|}{$100\% \pm 4\%$} &  \multicolumn{2}{c|}{$95\% \pm 2\%$}&  \multicolumn{2}{c|}{$92\% \pm 3\%$}  \\
   &  \multicolumn{2}{c|}{$522 \pm 102$}  &  \multicolumn{2}{c|}{$235 \pm 58 $} &  \multicolumn{2}{c|}{$42 \pm 6$}  &  \multicolumn{2}{c|}{$18 \pm 1.6$}\\
   \hline
   \multirow{2}{*}{30} &  \multicolumn{2}{c|}{$100\% \pm 0\%$}  &  \multicolumn{2}{c|}{$100\% \pm 4\%$} &  \multicolumn{2}{c|}{$96\% \pm 2\%$}&  \multicolumn{2}{c|}{$90\% \pm 4\%$}  \\
   &  \multicolumn{2}{c|}{$301 \pm 33$}  &  \multicolumn{2}{c|}{$139 \pm 30$} &  \multicolumn{2}{c|}{$36 \pm 4.1$}&  \multicolumn{2}{c|}{$14 \pm 1.2$} \\
   \hline\hline
    \end{tabular}
    }
\end{center}
\label{tab:inventory}
\end{table}

\subsubsection{Optimizing the drop rate $\eta$} \label{sec:opt-eta}
Finally, we explore boosting the performance of SEIU-MCB by optimizing $\eta$. The experiment setting and result can be found in \ref{ec:opt-eta-experiment}. The result shows the optimized $\eta$ improves both the PCS and the termination stage. 

\section{Conclusion and Future Directions} \label{sec:conclusion}
We study Ranking and Selection under input uncertainty where input data arrive in time-varying batches over time. A moving average estimator is proposed to aggregate the simulation outputs under different input models across time stages due to the limited simulation budget at each stage. We further propose the SEIU and SEIU-MCB procedures by respectively deriving the exact and asymptotic confidence bands in a substantially extended Sequential Elimination framework. We also derive the corresponding SEIU(IZ) and SEIU-MCB(IZ) procedures to extend our algorithms to the indifference zone setting. We analyze the impact of the drop rate of the moving average estimator on the procedures. Numerical results show the necessity of our proposed procedures under input uncertainty and the statistical validity to achieve the target PCS. The SEIU-MCB  procedure is highly efficient for practical use and can be accelerated by optimizing the drop rate.

There are several future research directions. First, other methods can be used to  update the input model and aggregate simulation outputs over time stages. Second, other R\&S procedures can be extended to the setting of streaming input data, especially when (asymptotic) normality of the aggregated performance estimate could be established. Third, other types of performance measure, such as quantiles, can be considered in data-driven ranking and selection problems.


\paragraph{ACKNOWLEDGMENT}{The authors gratefully acknowledge the support by the National Science Foundation under Grant DMS2053489 and Air Force Office of Scientific Research under Grant FA9550-19-1-0283 and Grant FA9550-22-1-0244.}

%
%
%


\bibliographystyle{plain} 
\bibliography{SEIU} 



\appendix
\section{Parameter Estimation} \label{ec:para-est}
We propose some heuristic methods for estimating the unknown parameters. To begin with, if the transformed input data $\{D_{t}\}$ and the simulation outputs $\{X_i\}$ are uniformly bounded, then $\nu_s$ and $\bar{\sigma}_{ij}$ can be computed using explicit bounds of their support. For example, if $X_i(\theta),X_j(\theta) \in [a, b]$ for all $\theta \in \Theta$, then $(b-a)$ would be a valid value of $\bar{\sigma}_{ij}$, but it tends to be conservative. An alternative is to approximate $\nu$ and $\bar{\sigma}_{ij}$ using the sample standard deviation of $D_{t}$ and $X_i - X_j$. For $\bar{L}_{ij}$, assuming that $\delta_{ij}(\cdot)$ is sufficiently smooth, by the mean value theorem and H\"older's inequality,
\begin{equation*}
|\delta_{ij}(\theta_1) - \delta_{ij}(\theta_2)| = \nabla \delta_{ij}(\tilde{\theta})^\intercal (\theta_1 - \theta_2) \le \sup_{\theta\in\Theta} \|\nabla\delta_{ij}(\theta) \|_\infty \|\theta_1 - \theta_2\|_1,
\end{equation*}
where $\nabla$ denotes the gradient operator and $\tilde{\theta}$ is a convex combination of $\theta_1$ and $\theta_2$. The issue is that $\sup_{\theta\in\Theta} \|\nabla\delta_{ij}(\theta) \|_\infty$ is both conservative and difficult to estimate. Empirically, we find that approximating $\bar{L}_{ij}$ with $\|\nabla \delta_{ij}(\hat{\theta}) \|_\infty$ usually preserves statistical validity for the procedure.In addition, $\nabla \mu_{i}(\theta)$ can be estimated via the likelihood ratio estimator
\begin{equation} \label{eq:est-L}
\widehat{\nabla \mu}_{i}(\theta) = \frac{1}{n_0} \sum_{r=1}^{n_0} X_{i,r}(\theta) \frac{\nabla_\theta f(\xi_{i,r}(\theta) \mid \theta)}{f(\xi_{i,r}(\theta) \mid \theta)}, \quad j=1,\ldots, d_s,
\end{equation}
where $n_0$ is an initial sample size, $f(\cdot | \theta)$ is the probability density (or the probability mass function) of $X$ under parameter $\theta$, and $\xi(\theta)$ denotes simulation sample generated from $P_\theta$. Then $\widehat{\nabla \delta}_{ij}(\theta) =\widehat{\nabla \mu}_{i}(\theta)  -\widehat{\nabla \mu}_{j}(\theta)   $Other methods for estimating the gradient can be found in \cite{fu2015handbook}. 

The impact of estimation errors on the procedures is discussed as follows. To begin with, underestimating $\nu_s, \sigma_{ij}$ or $\bar{L}_{ij}$ can cause premature stopping and undershooting the target PCS since the resulting confidence bands may fail to achieve desired coverage. However, as we shall see in Section \ref{sec:numerical}, this risk is mitigated by the conservatism of Bonferroni inequality. On the other hand, overestimating the parameters can lead to longer runtime and overshooting the target PCS, and thus we do not recommend estimators that are typically positively biased (e.g., $\sup_{\theta\in\Theta} \|\nabla\hat{\delta}_{ij}(\theta) \|_\infty$).

\section{Boosting Through $\eta$: More Analysis about Impact of Drop Rate} \label{ec:opt-eta}

 To minimize the impact caused by IU and SU, the analysis becomes difficult since we want to minimize a matrix-valued function in some sense. However, since our goal is to reduce the running time, which is often dominated by the difference of expected performance between the best and the second best design, we can then minimize the variance $\sigma^2_{bb'}$ to boost the procedure. Here $b$ and $b'$ denote the true best and the second best design, respectively. For this purpose, let $\sigma^2_{ij,\infty}$ denote the asymptotic variance of $\hat{\delta}_{ij,t}$. Then by applying Theorem \ref{thm:joint normality} with only two designs $i$ and $j$, we have 
$$
\sigma^2_{ij, \infty} := \lambda_{I,\eta} \nabla \delta_{ij}(\theta^c)^\intercal \bar{\Sigma}_D \nabla \delta_{ij}(\theta^c) + \lambda_{S,\eta} \bar{m}^{-1} \sigma^2_{ij}(\theta^c).
$$
Then, for any $\eta$ between $(0,1)$, $\sigma^2_{ij,\infty}$, the limiting variance for design $i$ and $j$,  satisfies

\begin{equation*}
\sigma^2_{ij,\infty} \propto \left\{\frac{\kappa}{1-\eta} + \frac{2\eta \ln \eta}{(1-\eta)^2} \right\},
\end{equation*}
where $\propto$ denotes ``is proportional to'', and
\begin{equation*}
\kappa := \frac{2 \sigma^2_{ij, IU} + \sigma^2_{ij, SU}}{2 \sigma^2_{ij, IU} },
\end{equation*}
in which $\sigma^2_{ij, IU}:= \nabla\delta_{ij}(\theta^c)^\intercal \bar{\Sigma}_D \nabla \delta_{ij}(\theta^c)$ and $\sigma^2_{ij, SU}:= \bar{m}^{-1} \sigma^2_{ij}(\theta^c)$. 
Notice that for two specific designs $i$ and $j$, the larger $\sigma^2_{ij,\infty}$, the harder for us to distinguish between them since we will get a larger confidence band $w_{ij,t}$. Hence, if we want elimination between $i$ and $j$ happens earlier, we should minimize $\sigma^2_{ij,\infty}$.  Minimizing $\sigma^2_{ij,\infty}$ is equivalent to solving the following optimization problem,
\begin{equation} \label{eq:min-eta}
\inf_{\eta \in (0,1)} \psi_\kappa(\eta) :=  \frac{\kappa}{1-\eta} + \frac{\eta \ln \eta}{(1-\eta)^2}, \quad \kappa \ge 1, 
\end{equation}
which does not have a closed-form solution. However, the following result characterizes some important properties of the optimal solution.
\begin{proposition} \label{prop:opt-eta}
For any fixed $\kappa > 1$, $\psi_\kappa$ is strictly convex in $(0,1)$ and always has a unique minimizer in $(0,1)$, denoted by $\eta^*(\kappa)$, which satisfies
\begin{equation} \label{eq:limit-eta}
 \eta^*(\kappa) \le \kappa^{-1}, \quad \lim_{\kappa \rightarrow 1} \eta^*(\kappa) = 1.
\end{equation}
Furthermore,
\begin{equation} \label{eq:inf-ratio}
\inf_{\kappa > 1} \frac{\psi_\kappa(\eta^*(\kappa))}{\psi_\kappa(0)} = \frac{1}{2}.
\end{equation}
\end{proposition}

Proposition \ref{prop:opt-eta} guarantees the existence and uniqueness of a minimizer, and the convexity of $\psi_\kappa$ makes it numerically easy to solve (\ref{eq:min-eta}). Moreover, connecting back to the definition of $\kappa$ yields the following interpretation.
\begin{enumerate}
\item[(i)]
If $\sigma^2_{ij,SU} \gg \sigma^2_{ij,IU}$, i.e., SU dominates IU, then $\kappa$ is large and the optimal $\eta$ is near 0. In other words, we should discard only a tiny portion of ``stale'' simulation outputs because averaging more outputs is the most effective way to reduce SU.
\item[(ii)]
If $\sigma^2_{ij,SU} \ll \sigma^2_{ij,IU}$, i.e., IU overshadows SU, then $\kappa \approx 1$ and the optimal $\eta$ would be near 1, meaning that it is more advantageous to discard most of the previous simulation output. Indeed, in the extreme case where there is no SU (i.e., the simulation output is $\mu_i(\hat{\theta}_t)$), there would be no need to keep anything but the latest output.
\end{enumerate}

The lower bound (\ref{eq:inf-ratio}) indicates that compared with blindly setting $\eta = 0$, choosing the optimal $\eta$ can reduce the asymptotic variance of $\hat{\delta}_{ij,t}$ by at most a factor of 2. Again, this happens when there is no SU, i.e., $\kappa=1$, since we have seen in Section \ref{sec:normality} that retaining only the latest output halves the asymptotic variance resulted from averaging all outputs.

\section{Supplement Material for Numerical Experiment} \label{ec:numerical}
\subsection{Quadratic problem} \label{ec:quad-prob}
The first problem is to minimize the expected value of a quadratic function $X_i = (i - \xi)^2$, where the true distribution of $\xi$ is Poisson with mean $\theta^c$. Let $\mathcal{I}$ be all the designs to be evaluated. Also let $\theta^c = 1 \in \mathcal{I}$ so that the best design is $b =\theta^c = 1$.

\subsection{Production-inventory problem} \label{ec:prod-inv-prob}
We also test the procedures on a production-inventory problem, where the objective function does not have a closed form but needs to be evaluated via simulation.

Suppose that we are running a capacitated production system and we want to minimize the expected total cost over a finite number of stages. The decision variable is the order-up-to level, i.e., the quantity we should fill up to once the inventory falls below that level. Please note this is an offline planning problem, where we select the best inventory policy after finishing the simulation.   Meanwhile, the production amount in each stage is capped. At the beginning of a stage, the amount produced in the previous stage arrives. Then, the demand is revealed and we fulfill the total demand (both backlog and current demand) to the best allowed by on-hand inventory, after which unfulfilled demand becomes the new backlog. Decision of the production amount is carried out at the end of the stage.
 
The variables are listed as follows: $i$ is the order-up-to level, $I_s$ is the inventory level at the end of the $s$th stage, $\xi_s$ is the demand at the $s$th stage, and $R_s$ is the production amount at the $s$th stage. Let $I_0 = i$ and $R_0 = 0$. Starting from $s=1$, the system dynamics evolve according to the following equations,
\begin{align*}
& I_{s} = I_{s-1} + R_{s-1} - \xi_s,\\
& R_s = \min\{R^*, (i - I_{s})^{+} \},
\end{align*}
where $a^+ := \max\{0, a\}$ and $R^*$ is the maximum production amount. Assume that the demand quantities are independent random variables, where each $\xi_s$ follows an exponential distribution with mean $\theta^c_s$. Let $c_H$ be the holding cost per unit and $c_B$ be the backlog cost per unit. Then, the cost at the $s$th stage is 
\begin{equation*}
c_s := c_H (R_{s-1} + I^+_s) + c_B I_s^-,
\end{equation*}
where $a^- := -\min\{a, 0\}$. The expected total cost over $S$ stages is 
\begin{equation*}
\mu_i(\theta^c) = \E\left(\sum_{s=1}^S c_s \right).
\end{equation*}
The goal is to select the optimal order-up-to quantity among candidate set $\mathcal{I}$.

\subsection{Experiment result for optimizing $\eta$} \label{ec:opt-eta-experiment}
For generality, we use random batch sizes in both the quadratic example in Section \ref{sec:IU control} and inventory-production example in Secton \ref{sec:Inv-Pro example}. We optimize $\eta$ to minimize the termination stage. Table \ref{tab:procedure-opt-eta} summarizes the average termination stage and the corresponding empirical PCS along with their 95\% confidence intervals, based on 100 independent runs of each procedure.

\begin{table}[ht]
\caption{Performance of SEIU-MCB: $\eta=0.1$ vs. optimized $\eta$.}
\begin{center}
\begin{tabular}{|c||cc|cc||cc|cc||}
\hline\hline
& \multicolumn{4}{c||}{Quadratic example } & \multicolumn{4}{c||}{Production-inventory example} \\
   \hline $\bar{k}$
   & \multicolumn{2}{c|}{$\eta=0.1$} & \multicolumn{2}{c||}{Optimized $\eta$} &\multicolumn{2}{c|}{$\eta=0.1$} & \multicolumn{2}{c||}{Optimized $\eta$}  \\
    \hline\hline
    \multirow{2}{*}{10} & \multicolumn{2}{c|}{ $94\% \pm 3\%$}& \multicolumn{2}{c||}{ $91\% \pm 4\%$}& \multicolumn{2}{c|}{ $93\% \pm 3\%$} & \multicolumn{2}{c||}{ $89\% \pm 4\%$}\\ & \multicolumn{2}{c|}{ $47 \pm 6.4$}& \multicolumn{2}{c||}{ $35 \pm 4.8$}& \multicolumn{2}{c|}{ $55 \pm 4.7$} & \multicolumn{2}{c||}{ $37 \pm 3.8$}\\
    \hline \hline
    \multirow{2}{*}{20} & \multicolumn{2}{c|}{ $95\% \pm 4\%$}& \multicolumn{2}{c||}{$ 92\% \pm 6\%$}& \multicolumn{2}{c|}{ $96\% \pm 3\%$} & \multicolumn{2}{c||}{ $94\% \pm 3\%$}\\ & \multicolumn{2}{c|}{ $25 \pm 2.7$}& \multicolumn{2}{c||}{ $20 \pm 3.7$}& \multicolumn{2}{c|}{ $28 \pm 2.4$} & \multicolumn{2}{c||}{ $22 \pm 2.1$}\\
    \hline \hline
     \multirow{2}{*}{30} & \multicolumn{2}{c|}{ $98\% \pm 3\%$}& \multicolumn{2}{c||}{ $96\% \pm 4\%$}& \multicolumn{2}{c|}{ $98\% \pm 1\%$} & \multicolumn{2}{c||}{ $98\% \pm 1\%$}\\ & \multicolumn{2}{c|}{ $20 \pm 2.6$}& \multicolumn{2}{c||}{ $16 \pm 2.5$}& \multicolumn{2}{c|}{ $21 \pm 1.5$}& \multicolumn{2}{c||}{ $17 \pm 1.3$}\\
    \hline \hline
\end{tabular}
\end{center}
\label{tab:procedure-opt-eta}
\end{table}

From Table \ref{tab:procedure-opt-eta}, we can conclude that the drop rate $\eta$ balances the trade off between PCS and run time. In almost all settings, the optimal drop rate results in less run time of the procedure while achieving a relatively lower PCS. This is consistent with our intuition that less total amount of data (corresponding to earlier termination stage) increases input uncertainty and hence results in a lower PCS. It is worth noting that although the optimal drop rate results in a lower empirical PCS, it still achieves the target PCS. From a practical viewpoint, a good selection of the drop rate will boost the efficiency by reducing the run time while still satisfying the probabilistic guarantee. \\

\section{Proofs for SEIU}
\textbf{Proof of Proposition \ref{prop:cb}}
Suppress $i,j$ from the index for convenience. Letting
\begin{align*}
&\bar{\delta}_t := \E[\hat{\delta}_{t} \mid \hat{\theta}_1, \ldots, \hat{\theta}_t] =  [M(t) - M(t_\eta)]^{-1}\sum_{\ell=t_\eta+1}^t m(\ell) \delta(\hat{\theta}_\ell),\\
&Y_r= Y_{ij,r} := X_{i,r} - X_{j,r}, 
\end{align*}
we have
\begin{equation} \label{eq:ub}
\mathbb{P}\left\{|\hat{\delta}_{t} -\delta(\theta^c)| > x+y \right\} \le \mathbb{P}\left\{|\hat{\delta}_{t} -\bar{\delta}_t| > x \right\} + \mathbb{P}\left\{|\bar{\delta}_{t} -\delta(\theta^c)| > y \right\}
\end{equation}
For the first term on the RHS of (\ref{eq:ub}), 
\begin{align*}
\mathbb{P}\left\{|\hat{\delta}_{t} -\bar{\delta}_t| > x \right\} 
&= \E\left[ \mathbb{P}\left\{ \bigg| [M(t) - M(t_\eta)]^{-1} \sum_{\ell=t_\eta+1}^t\sum_{r=1}^{m(\ell)} [Y_r(\hat{\theta}_\ell) - \delta(\hat{\theta}_\ell)] \bigg| > x \,\middle\vert \,  \hat{\theta}_1, \ldots, \hat{\theta}_t \right\} \right] \\
&\le 2 \exp\left\{- \frac{[M(t) - M(t_\eta)]x^2}{2 {\bar{\sigma}}^2} \right\},
\end{align*}
where the inequality holds since conditioned on $\hat{\theta}_1, \ldots, \hat{\theta}_t$, $\{Y_r(\hat{\theta}_\ell) - \delta(\hat{\theta}_\ell)\}_r$ are independent zero-mean $subG(\bar{\sigma}^2)$ random variables. For the second term on the RHS of (\ref{eq:ub}),
\begin{align*}
\mathbb{P}\left\{|\bar{\delta}_{t} -\delta(\theta^c)| > y \right\} &\le \mathbb{P}\left\{ \bigg| [M(t) - M(t_\eta)]^{-1} \sum_{\ell=t_\eta+1}^t m(\ell) [\delta(\hat{\theta}_\ell) - \delta(\theta^c)] \bigg| > y  \right\} \\
& \le \mathbb{P}\left\{\sum_{\ell=t_\eta+1}^t m(\ell) |\delta(\hat{\theta}_\ell) - \delta(\theta^c)| > [M(t) - M(t_\eta)] y \right\} \\
& \le \mathbb{P}\left\{\sum_{\ell=t_\eta+1}^t m(\ell) \|\hat{\theta}_\ell - \theta^c\|_1 > [M(t) - M(t_\eta)]\frac{y}{\bar{L}} \right\} \\
& =  \mathbb{P}\left\{\sum_{\ell=t_\eta+1}^t m(\ell)  \sum_{s=1}^S\sum_{j=1}^{d_s} |\hat{\theta}_{s,j,\ell} - \theta^c_{s,j}| > [M(t) - M(t_\eta)] \frac{y}{\bar{L}} \right\} \\
& \le \sum_{s=1}^S\sum_{j=1}^{d_s} \mathbb{P}\left\{\sum_{\ell=t_\eta+1}^t m(\ell) |\hat{\theta}_{s,j,\ell} - \theta^c_{s,j}| > [M(t) - M(t_\eta)] \frac{y}{d\bar{L}} \right\} \\
&:= \sum_{s=1}^S\sum_{j=1}^{d_s} p_{s,j},
\end{align*}
where $\hat{\theta}_{s,j,\ell}$ denotes the $j$th coordinate of $\hat{\theta}_{s,\ell}$. Note that for any nonnegative sequence $\{\omega_{s,\ell}\}_{\ell=t_\eta+1}^t$ that satisfies $\sum_{\ell=t_\eta+1}^t \omega_{s,\ell} = 1$, we have
\begin{align}
p_{s,j} &\le \mathbb{P} \left\{ \bigcup_{\ell = t_\eta+1}^t \left\{m(\ell) |\hat{\theta}_{s,j,\ell} - \theta^c_{s,j}| > \omega_{s,\ell}  [M(t) - M(t_\eta)] \frac{y}{d\bar{L}}  \right\}  \right\} \notag \\
& \le \sum_{\ell=t_\eta+1}^t \mathbb{P} \left\{m(\ell) |\hat{\theta}_{s,j,\ell} - \theta^c_{s,j}| > \omega_{s,\ell}  [M(t) - M(t_\eta)] \frac{y}{d\bar{L}} \right\} \\
&\le \sum_{\ell=t_\eta+1}^t 2\exp\left\{- \frac{N_s(\ell) \omega^2_{\ell} [M(t) - M(t_\eta)]^2 y^2}{2d^2\bar{L}^2 \nu_s^2 [m(\ell)]^2} \right\}, \notag
\end{align}
where the second inequality follows from the fact that $\hat{\theta}_{s,j,\ell} \sim subG\left(\nu_s^2 / N_s(\ell)\right)$. Take 
\begin{equation*}
\omega_{s,\ell} = \frac{m(\ell) / \sqrt{N_s(\ell)}}{\gamma_{s,\eta}(t)}
\end{equation*}
and we further have
\begin{equation} \label{eq:p_sj}
p_{s,j} \le 2(t-t_\eta) \exp\left\{- \frac{[M(t) - M(t_\eta)]^2 y^2}{2d^2 \bar{L}^2 \nu_s^2 \gamma_{s,\eta}^2(t)} \right\}.
\end{equation}
Plugging (\ref{eq:p_sj}) back into (\ref{eq:ub}) completes the proof. \hfill  $\blacksquare$

\textbf{Proof of Lemma {\ref{lemma:cb}}}
Similar to the proof of Proposition \ref{prop:cb}, let
\begin{equation*}
\bar{\delta}_{i,t} := \E[\hat{\delta}_{i,t} \mid \hat{\theta}_1, \ldots, \hat{\theta}_t] = [M_i(t) - M_i(t_\eta)]^{-1}\sum_{\ell=t_\eta+1}^t m_i(\ell) \delta_i(\hat{\theta}_\ell)
\end{equation*}
and we have 
\begin{align*}
\mathbb{P} \left\{\bigcup_{ i<j } \left\{|\hat{\delta}_{ij,t} - \delta_{ij}(\theta^c)| > x_{ij} + y_{ij} \right\}  \right\}
\le \sum_{ i<j } \mathbb{P}\left\{|\hat{\delta}_{i,t} - \bar{\delta}_{i,t}| >x_{ij} \right \} +
\mathbb{P} \left\{\cup_{ i<j } \mathcal{E}_{ij,t} \right \},
\end{align*}
where $\mathcal{E}_{ij,t}:= \left\{|\bar{\delta}_{ij,t} - \delta_{ij}(\theta^c)| > y_{ij} \right\}$. By the same conditioning argument in the proof of Proposition \ref{prop:cb} and a sub-Gaussian tail bound, we have
\begin{align} \label{eq:su-cb}
\sum_{  i<j } \mathbb{P}\left\{|\hat{\delta}_{ij,t} - \bar{\delta}_{ij,t}| >x_{ij} \right   \} \le \sum_{  i<j } 2 \exp\left\{- \frac{[M(t) - M(t_\eta)]x_{ij}^2}{2  \bar{\sigma_{ij}}^2} \right\},
\end{align}
which gives the first part of the desired bound. For $\mathcal{E}_{ij,t}$, in light of the Lipschitz continuity of $\delta_{ij}$,
\begin{align*}
\mathcal{E}_{ij,t} 
&\subseteq \left\{\sum_{\ell=t_\eta+1}^t m(\ell) \|\hat{\theta}_\ell - \theta^c\|_1 > [M(t) - M(t_\eta)]\frac{y_{ij}}{\bar{L}_{ij}}  \right\} \\
& = \left\{ \sum_{j=1}^{d} \sum_{\ell=t_\eta+1}^t m(\ell) |\hat{\theta}_{s,j,\ell} - \theta^c_{s,j}| > [M(t) - M(t_\eta)]\frac{y_{ij}}{\bar{L}_{ij}} \right\} \\
& \subseteq \bigcup_{s=1}^S\bigcup_{j=1}^{d_s} \left\{ \sum_{\ell=t_\eta+1}^t m(\ell) |\hat{\theta}_{s,j,\ell} - \theta^c_{s,j}| > [M(t) - M(t_\eta)]\frac{y_{ij}}{d\bar{L}_{ij}} \right\}  \\
& \subseteq \bigcup_{s=1}^S\bigcup_{j=1}^{d_s} \bigcup_{\ell=t_\eta+1}^t \left\{|\hat{\theta}_{s,j,\ell} - \theta^c_{s,j}| > \frac{\omega_{s,\ell} [M(t) - M(t_\eta)] y_{ij}}{m(\ell) d \bar{L}_{ij}} \right\},
\end{align*}
where $\omega_{s,\ell} = m(\ell)/ \left[\sqrt{N_s(\ell)} \gamma_{s,\eta}(t)\right]$ satisfies $\sum_{\ell = t_\eta+1}^t \omega_{s,\ell}=1$. Plugging in the value of $\omega_{s,\ell}$ and we further have
\begin{align*}
\cup_{  i<j } \mathcal{E}_{ij,t} \subseteq 
  \bigcup_{s=1}^S\bigcup_{j=1}^{d_s} \bigcup_{\ell=t_\eta+1}^t \left\{|\hat{\theta}_{s,j,\ell} - \theta^c_{s,j}| > \min_{  i<j } \left\{\frac{[M(t) - M(t_\eta)] y_{ij}}{d\bar{L}_{ij}\sqrt{N_s(\ell)} \gamma_{ \eta}(t)} \right\} \right\},
\end{align*}
and it follows from the fact $\hat{\theta}_{s,j,\ell}\sim subG(\bar{\nu}^2 / N_s(\ell))$ that
\begin{align} \label{eq:iu-cb}
\mathbb{P}(\cup_{  i<j } \mathcal{E}_{ij,t}    ) \le 2(t-t_\eta)  \sum_{s=1}^S d_s \max_{  i<j } \left\{ \exp\left\{- \frac{[M(t) - M(t_\eta)]^2 y_{ij}^2}{2 \nu_s^2 d^2 \bar{L}^2_{ij} \gamma^2_{s,\eta}(t)} \right\}\right\}.
\end{align}
Finally, combining (\ref{eq:su-cb}) and (\ref{eq:iu-cb}) yields the result. \hfill  $\blacksquare$

\textbf{{Proof of Theorem \ref{thm:SEIU-guarantee}}}
According to (\ref{eq:SE}), it suffices to show that

\begin{equation} \label{eq:cb-validity}
 \mathbb{P}\left\{\bigcup_{  i<j } \left\{|\hat{\delta}_{ij,t} - \delta_{ij}(\theta^c)| > w_{ij,t}\right\}     \right\}  \le \dfrac{6\alpha}{\pi^2t^2},
\end{equation}

which would imply that the confidence bands $\{w_{ij,t}\}$ jointly achieve at least $1-\alpha$ coverage probability. Apply Lemma \ref{lemma:cb} to obtain

\begin{align}
&  \mathbb{P}\left\{\bigcup_{  i<j } \left\{|\hat{\delta}_{ij,t} - \delta_{ij}(\theta^c)| > w_{ij,t}\right\}    \right\} 
=  \mathbb{P}\left\{\bigcup_{  i<j } \left\{|\hat{\delta}_{ij,t} - \mu_{ij}(\theta^c)| > u_{ij,t} + v_{ij,t}\right\}    \right\}  \notag \\
\le& \sum_{  i<j } 2 \exp\left\{- \frac{[M(t) - M(t_\eta)]u_{ij,t}^2}{2  \bar{\sigma_{ij}}^2} \right\} \tag{$\dagger$} \\
 & +  2(t-t_\eta)   d \max_{  i<j } \left\{\exp\left\{- \frac{[M(t) - M(t_\eta)]^2 v_{ij,t}^2}{2 \nu_s^2 d^2 \bar{L}^2_{ij} \gamma^2_{s,\eta}(t)} \right\}\right\} \tag{$\dagger\dagger$} 
\end{align}
We check $(\dagger)$ and $(\dagger\dagger)$ as follows.
\begin{align*}
(\dagger) &= \sum_{  i<j } 2 \exp \left\{-\frac{M(t) - M(t_\eta)}{2\bar{\sigma}_{ij}^2} \cdot
\frac{4 \bar{\sigma}_{ij}^2 \ln{\left(\sqrt{\frac{  K (  K -1)\pi^2}{3\alpha}}t\right)}}{M(t) - M(t_\eta)}\right\} \\
&= \sum_{  i<j} 2\exp\left\{-\ln\left(\frac{  K(  K -1)\pi^2 t^2}{3\alpha} \right) \right\} \\
&=    K ( K-1)\cdot \frac{3\alpha}{   K( K-1)\pi^2 t^2}  = \frac{3\alpha}{\pi^2t^2}.
\end{align*}
For ($\dagger\dagger$), first note that
\begin{align*}
\exp\left\{- \frac{[M(t) - M(t_\eta)]^2 v_{ij,t}^2}{2 \nu_s^2 d^2 \bar{L}^2_{ij} \gamma^2_{s,\eta}(t)} \right\} &\le \exp\left\{- \ln\left(\frac{2 S d_s\pi^2 (1-\eta)t^3}{3\alpha} \right) \right\} = \frac{3\alpha}{2 Sd_s  \pi^2(1-\eta)t^3},
\end{align*}
and it follows that
\begin{align*}
(\dagger\dagger) &= 2(t-t_\eta) \sum_{s=1}^S d_s \frac{3\alpha}{2 Sd_s  \pi^2(1-\eta)t^3} \\
&=   \frac{3\alpha}{\pi^2 t^2  } .
\end{align*}
Combining $(\dagger)$ and $(\dagger\dagger)$ gives (\ref{eq:cb-validity}).
It remains to show the upper bound on $\E[\tilde{\tau}]$. We claim that if all $\delta_{ij}(\theta^c)$ fall within their corresponding confidence bands $[\hat{\delta}_{ij,t} - w_{ij,t}, \hat{\delta}_{ij,t} + w_{ij,t}] $, then any suboptimal design $i \neq b$ must be eliminated prior to stage $\min_{j\neq i} \tilde{\tau}_{ij}$. To see this, note that $\tilde{\tau}_{ij}$ is how long design $i$ can last before getting eliminated by design $j$. Thus, $\min_{j \neq i} \tilde{\tau}_{ij}$ is the longest survival time of design $i$. Furthermore, the procedure must terminate before stage $\tilde{\tau}_* := \max_{i \neq b} \min_{j \neq i} \tilde{\tau}_{ij}$. In other words, if the procedure has not terminated by some stage $t > \tilde{\tau}_*$, then at least one design's confidence band has failed to cover its true mean at stage $t$, i.e., 
\begin{equation*}
\{\tilde{\tau} > t\} \subseteq \bigcup_{\substack{i<j}} \left\{|\hat{\delta}_{i,t} - \delta_{ij}(\theta^c)| > w_{ij,t} \right\}, \quad \forall t > \tilde{\tau}_*.
\end{equation*}
Let $\mathcal{E}_t :=\bigcup\limits_{\substack{i<j}} \left\{|\hat{\delta}_{ij,t} - \delta_{ij}(\theta^c)| > w_{ij,t} \right\}$,

\begin{align*}
\E[\tilde{\tau}] & = \sum_{t=0}^\infty \mathbb{P}(\tilde{\tau} > t) = \sum_{t=0}^{\tilde{\tau}_* - 1} \mathbb{P}(\tilde{\tau} > t) + \sum_{t = \tilde{\tau}_*}^\infty \mathbb{P}(\tilde{\tau} > t)  \\
&\le \tilde{\tau}_* + \sum_{t=\tilde{\tau}_*}^\infty \mathbb{P}(\mathcal{E}_t) \le \tilde{\tau}_* + \sum_{t=1}^\infty \mathbb{P}(\mathcal{E}_t) \le \tilde{\tau}_* + \alpha,
\end{align*}
where the last inequality follows from (\ref{eq:cb-validity}). \hfill $\blacksquare$

\textbf{Proof of Corollary \ref{co:SEIU}}

According to our SE framework in Section \ref{sec:SE}, we know that with probability as least $1-\alpha$, the confidence bands achieve perfect coverage. Hence at any stage $t$, the best design $b$ remains in the remaining set $\mathcal{S}$.

Now, with probability at least $1-\alpha$, we have 
\begin{align*}
    \mu_{i^*} (\theta^c)- \mu_b(\theta^c) \ge \hat{\delta}_{i^* b, t} - w_{i^*b,t} \ge - w_{i^*b,t} \ge - \max_{j \in \mathcal{S}, j\neq i*}w_{i*j,t}.
\end{align*}
Hence the current design with the largest sample mean is  $\epsilon_t$ optimal, where $\epsilon_t = \max_{j \in \mathcal{S}, j\neq i*}w_{i*j,t}$.

\textbf{Proof of Corollary \ref{co:IZ}}
The validity of the SEIU(IZ) is the result of Theorem \ref{thm:SEIU-guarantee} and Corollary \ref{co:SEIU}. For the termination time, notice that for $t>\tau$ and any $j \in \mathcal{S}, j\neq i^*$
\begin{align*}
    w_{i*j,t} \le \max_{  i\neq j } {w}_{ij,t} < \epsilon
\end{align*}
Hence, $ \max\limits_{j \in \mathcal{S}, j\neq i^*} w_{i^*j,t} < \epsilon$. The procedure must terminate before $\tau$.

\section{Proofs for SEIU-MCB}
We establish a series of Lemmas for proving Theorem \ref{thm:joint normality}.
\begin{lemma}[The Lindeberg-Feller Theorem ( Proposition 2.2.7 in \cite{van2000asymptotic}] \label{lemma:Lindberg}
For each $n$ let $Y_{n, 1}, \ldots$, $Y_{n, k_{n}}$ be independent random vectors with finite variances such that
\begin{enumerate}
    \item [(i)] $ \sum_{i=1}^{k_{n}} \operatorname{Cov} Y_{n, i}  \rightarrow \Sigma,$
    \item [(ii)] $\sum_{i=1}^{k_{n}} \mathrm{E}\left\|Y_{n, i}\right\|^{2} 1\left\{\left\|Y_{n, i}\right\|>\varepsilon\right\}  \rightarrow 0, \quad \text { every } \varepsilon>0.$
\end{enumerate}
Then the sequence $\sum_{i=1}^{k_{n}}\left(Y_{n, i}-\mathrm{E} Y_{n, i}\right)$ converges in distribution to a normal $N_s(0, \Sigma)$ distribution.
\end{lemma}

\begin{lemma} \label{lemma:CLT}
Let $\{X_n\}$ be independent random vectors with $\E X_n = 0$. If $X_n \Rightarrow X$ and $\operatorname{Cov} X_n \rightarrow \operatorname{Cov} X$ as $n \rightarrow \infty$, then for any $\eta \in [0,1)$,
$$(n-n_\eta)^{-1/2}\sum_{i=n_\eta+1}^n X_i \Rightarrow \mathcal{N}(0, \operatorname{Cov} X) \quad \text{ as } n\rightarrow \infty.$$
\end{lemma}

\textbf{Proof of Lemma \ref{lemma:CLT}.}
Let $Y_{n,i} := 0$ if $i \le n_\eta$ and $X_i / \sqrt{n-n_\eta}$ otherwise. We apply Lemma \ref{lemma:Lindberg} to $Y_{n,i}$. Condition (i) is satisfied since 
$$\sum_{i=1}^n \operatorname{Cov} Y_{n,i} = \sum_{i=n_\eta+1}^n \operatorname{Cov} X_i / (n-n_\eta) \rightarrow \operatorname{Cov} X.$$ 
In addition, by a generalized Dominated Convergence Theorem (note that the integrand is dominated by $||X_n||^2$ and $\operatorname{Cov} X_n \rightarrow \operatorname{Cov} X$),
$$\E\left[||X_n||^2 \mathbbm{1}_{\{||X_n|| > \epsilon \sqrt{n-n_\eta}\}}\right] \rightarrow 0, \quad \text{as } n \rightarrow \infty$$
Thus, as $n \rightarrow \infty$,
\begin{align*}
\sum_{i=1}^n \E[||Y_{n,i}||^2 \mathbbm{1}_{\{||Y_{n,i}||>\epsilon\}}] &= \frac{1}{n-n_\eta} \sum_{i=n_\eta+1}^n\E[||X_i||^2 \mathbbm{1}_{\{||X_i|| > \epsilon \sqrt{n-n_\eta}\}}] \rightarrow 0,
\end{align*}
which verifies condition (ii), and the result follows. \hfill $\blacksquare$

The next few lemmas revolve around the convergence properties of
\begin{equation} \label{eq:phi}
\phi_{s,t}(k) := \sum_{\ell=k}^t m(\ell) / N_s(\ell), \quad  k\le t,
\end{equation}
which plays a crucial role in the proof of Theorem \ref{thm:joint normality}.

\begin{lemma}\label{lemma:conv-gamma}
Under Assumption \ref{assump:CLT}, $\phi_{t}(t_\eta) \rightarrow -\bar{m} \ln \eta / \bar{n}$ as $t \rightarrow \infty$.
\end{lemma}
\textbf{Proof of Lemma \ref{lemma:conv-gamma}}
 We drop the index $s$ for simplicity. Note that
\begin{align*}
\left| \phi_t(t_\eta) - \sum_{\ell=t_\eta}^t \frac{\bar{m}}{\bar{n} \ell} \right| \le \left| \phi_t(t_\eta) - \sum_{\ell=t_\eta}^t \frac{\bar{m}}{N(\ell)} \right| + \left| \sum_{\ell=t_\eta}^t \frac{\bar{m}}{N(\ell)} - \sum_{\ell=t_\eta}^t \frac{\bar{m}}{\bar{n} \ell} \right|.
\end{align*}
For the first term in the RHS, 
\begin{align*}
\left| \phi_t(t_\eta) - \sum_{\ell=t_\eta}^t \frac{\bar{m}}{N_s(\ell)} \right| &= \left| \sum_{\ell=t_\eta}^t \frac{m(\ell) -\bar{m}}{N(\ell)} \right| \le \frac{1}{N(t_\eta)} \left| \sum_{\ell=t_\eta}^t m(\ell) - (t-t_\eta+1) \bar{m} \right| \\
&= \frac{t-t_\eta+1}{N(t_\eta)} \frac{1}{t-t_\eta+1} \left| \sum_{\ell=t_\eta}^t m(\ell) - (t-t_\eta+1) \bar{m} \right| \rightarrow 0, \quad \text{as } t \rightarrow \infty.
\end{align*}
For the second term,
\begin{align*}
\left| \sum_{\ell=t_\eta}^t \frac{\bar{m}}{N(\ell)} - \sum_{\ell=t_\eta}^t \frac{\bar{m}}{\bar{n}\ell} \right| =
\frac{\bar{m}}{\bar{n}} \left| \sum_{\ell=t_\eta}^t \frac{\bar{n}\ell - N(\ell)}{N(\ell)\ell} \right| \le \frac{\bar{m}}{\bar{n}} \frac{1}{N(t_\eta)} \left| \sum_{\ell=t_\eta}^t \left( \bar{n} - \frac{N(\ell)}{\ell} \right) \right| \rightarrow 0, \quad \text{as } t \rightarrow \infty.
\end{align*}
Therefore, we have 
\begin{equation*}
\lim_{t\rightarrow \infty} \phi_t(t_\eta) = \frac{\bar{m}}{\bar{n}} \lim_{t\rightarrow \infty} \sum_{\ell=t_\eta}^t \frac{1}{\ell} = \frac{\bar{m}}{\bar{n}} \lim_{t\rightarrow \infty}  \int_{t_\eta}^t \frac{1}{x} dx = -\frac{\bar{m} \ln\eta}{
\bar{n}},
\end{equation*} 
which shows the desired limit. \hfill $\blacksquare$

\begin{lemma} \label{lemma:gamma}
For $k \le t$,
\begin{equation*}
\sum_{\ell=1}^k \phi_t^2(\ell) n(\ell) = \sum_{\ell=1}^{k-1} \frac{m^2(\ell)}{N(\ell)} + 2 \sum_{\ell=1}^{k-2} \frac{M(\ell)m(\ell+1)}{N(\ell+1)} + 2M(k-1) \phi_t(k) + N(k) \phi_t^2(k).
\end{equation*}
\end{lemma}

\textbf{Proof of Lemma \ref{lemma:gamma}}
By expanding $\phi_t(\ell)$ into $m(\ell)/N(\ell) + \phi_t(\ell+1)$ recursively, we have
\begin{align*}
\sum_{\ell=1}^k \phi_t^2(\ell) n(\ell) &= \phi_t^2(1) n(1) + \sum_{\ell=2}^k \phi_t^2(\ell) n(\ell) \\
&= n(1) \left[\frac{m(1)}{N(1)} + \phi_t(2) \right]^2 + \sum_{\ell=2}^k \phi_t^2(\ell) n(\ell) \\
&= \frac{m^2(1)}{N(1)} + 2m(1) \phi_t(2) + N(2) \phi_t^2(2) + \sum_{\ell=3}^k \phi_t^2(\ell) n(\ell) \\
&= \frac{m^2(1)}{N(1)} + 2m(1) \phi_t(2) + N(2) \left[\frac{m(2)}{N(2)} + \phi_t(3) \right]^2 + \sum_{\ell=3}^k \phi_t^2(\ell) n(\ell) \\
&= \cdots \\
&= \sum_{\ell=1}^{k-1} \frac{m^2(\ell)}{N(\ell)} + 2 \sum_{\ell=1}^{k-1} m(\ell) \phi_t(\ell+1) + N(k) \phi_t^2(k). \tag{*}
\end{align*}
Similarly, expanding the second term in (*) gives
\begin{align*}
\sum_{\ell=1}^{k-1} m(\ell) \phi_t(\ell+1) &= 
m(1) \left[\frac{m(2)}{N(2)} + \phi_t(3)\right] + \sum_{\ell=2}^{k-1} m(\ell)\phi_t(\ell+1)\\
&= \frac{m(1)m(2)}{N(2)} + M(2)\phi_t(3) + \sum_{\ell=3}^{k-1} m(\ell)\phi_t(\ell+1)\\
&= \frac{M(1)m(2)}{N(2)} + \frac{M(2)m(3)}{N(3)}  + M(3)\phi_t(4) + \sum_{\ell=4}^{k-1} m(\ell)\phi_t(\ell+1)\\
&=\cdots\\
&=\sum_{\ell=1}^{k-2} \frac{M(\ell)m(\ell+1)}{N(\ell+1)} + M(\ell-1) \phi_t(k). \tag{**}
\end{align*}
Plugging (**) into (*) yields the final result. \hfill $\blacksquare$

\begin{lemma} \label{lemma:beta}
Let Assumption \ref{assump:CLT} hold. Then,
\begin{equation} \label{conv-cond}
\left[M(t) - M(t_\eta)\right]^{-1} [\phi^2_{ t} (t_\eta+1) N_s(t_\eta) + \sum_{\ell=t_\eta+1}^t \phi^2_{t}(\ell)n_s(\ell)] \rightarrow \beta_{s,\eta}, \quad \text{as } t \rightarrow \infty,
\end{equation}
where
\begin{equation} \label{eq:beta}
\beta_{s,\eta} = \frac{\bar{m}}{\bar{n}_s}\left(2 + \frac{2\eta \ln\eta}{1-\eta} \right).
\end{equation}
\end{lemma}

\textbf{Proof of Lemma \ref{lemma:beta}}
Fix $s$ so that it can be suppressed from the indices. Define
\begin{equation}
\tilde{\beta}_\eta(t) := \left[M(t) - M(t_\eta)\right]^{-1} [\phi^2_{t} (t_\eta+1) N(t_\eta) + \sum_{\ell=t_\eta+1}^t \phi^2_{t}(\ell)n(\ell)].
\end{equation}
The goal is to show that $\tilde{\beta}_\eta(t)  \rightarrow \beta_\eta$ as $t \rightarrow \infty$. In light of Lemma \ref{lemma:gamma}, 
\begin{align*}
\sum_{\ell=t_\eta+1}^t \phi_t^2(\ell) n(\ell) &= \sum_{\ell=t_\eta}^t \frac{m^2(\ell)}{N(\ell)} + 2\sum_{\ell=t_\eta+1}^{t-2} \frac{M(\ell)m(\ell+1)}{N(\ell)} \\
&+ 2M(t-1)\phi_t(t) - 2M(t_\eta-1) \phi_t(t_\eta) \\
&+ N(t) \phi_t^2(t) - N(t_\eta) \phi_t^2(t_\eta)
\end{align*}
Plugging it into $\tilde{\beta}_\eta(t)$ and we have
\begin{align*}
\tilde{\beta}_\eta(t) &= \frac{t - t_\eta}{M(t) - M(t_\eta)}  \cdot \frac{1}{t-t_\eta} \cdot \left\{ \phi_t^2(t_\eta+1) N(t_\eta) + \sum_{\ell=t_\eta}^t \frac{m^2(\ell)}{N(\ell)}  \right. \\
&+ 2\sum_{\ell=t_\eta+1}^{t-2} \frac{M(\ell)m(\ell+1)}{N(\ell)} + 2M(t-1)\phi_t(t) - 2M(t_\eta-1) \phi_t(t_\eta) \\
&+ N(t) \phi_t^2(t) - N(t_\eta) \phi_t^2(t_\eta) \left. \right\}
\end{align*}
We will compute the limit of each term. To begin with,
\begin{equation*}
\lim_{t \rightarrow \infty}  \frac{t - t_\eta}{M(t) - M(t_\eta)} = \frac{1}{\bar{m}}.
\end{equation*}
Next, since $\lim_{t\rightarrow \infty} \phi_t(t_\eta)$ exists by Lemma \ref{lemma:conv-gamma}, the first and last terms in the curly bracket cancel out in the limit:
\begin{equation*}
\lim_{t \rightarrow \infty} \frac{1}{t-t_\eta} \left[\phi_t^2(t_\eta+1) N(t_\eta) - N(t_\eta) \phi_t^2(t_\eta) \right] = 0.
\end{equation*}
Meanwhile, since $m(t)$ is bounded and $N(t) \rightarrow \infty$ as $t \rightarrow \infty$,
\begin{align*}
&\lim_{t \rightarrow \infty} \frac{1}{t-t_\eta} \sum_{\ell=t_\eta}^t \frac{m^2(\ell)}{N(\ell)}  = 0,  \\
&\lim_{t \rightarrow \infty} \frac{1}{t-t_\eta} 2M(t-1) \phi_t(t) = \frac{2\bar{m}}{1 - \eta} \lim_{t \rightarrow \infty} \frac{m(t)}{N(t)} = 0 \\
&\lim_{t \rightarrow \infty} \frac{1}{t-t_\eta} N(t) \phi_t^2(t) = \frac{\bar{n}}{1-\eta} \lim_{t \rightarrow \infty} \frac{m(t)}{N(t)} = 0.
\end{align*}
For the remaining terms, we further have
\begin{equation*}
\lim_{t \rightarrow \infty} \frac{1}{t-t_\eta} 2M(t_\eta-1) \phi_t(t_\eta) = \frac{\bar{m}^2}{\bar{n}} \frac{2\eta \ln\eta}{1 -\eta}.
\end{equation*}
Also, 
\begin{equation*}
\lim_{t \rightarrow \infty} \frac{1}{t-t_\eta} \cdot 2\sum_{\ell=t_\eta+1}^{t-2} \frac{M(\ell)m(\ell+1)}{N(\ell)} = \frac{2\bar{m}^2}{\bar{n}},
\end{equation*}
which can be seen as follows. Since $M(t) / N(t) \rightarrow \bar{m} / \bar{n}$ as $t \rightarrow \infty$, for any $\epsilon > 0$ there exists $T > 0$ such that $|M(t) / N(t) -  \bar{m} / \bar{n}| < \epsilon$ for all $t > T$. Hence, for all $t > T$,
\begin{align*}
\left(\frac{\bar{m}}{\bar{n}} - \epsilon \right) \frac{1}{t-t_\eta}  \sum_{\ell=t_\eta+1}^{t-2}  m(\ell+1) &\le  \frac{1}{t-t_\eta} \sum_{\ell=t_\eta+1}^{t-2} \frac{M(\ell)m(\ell+1)}{N(\ell)}  \\
&\le \left(\frac{\bar{m}}{\bar{n}} + \epsilon \right) \frac{1}{t-t_\eta}  \sum_{\ell=t_\eta+1}^{t-2}  m(\ell+1),
\end{align*}
where taking $t\rightarrow \infty$ and $\epsilon \rightarrow 0$ confirms the limit. Putting everything together,
\begin{equation*}
\lim_{t\rightarrow \infty}\tilde{\beta}_\eta(t) = \frac{1}{\bar{m}} \left(\frac{2\bar{m}^2}{\bar{n}} +\frac{\bar{m}^2}{\bar{n}} \frac{2\eta \ln\eta}{1-\eta}\right) = \frac{\bar{m}}{\bar{n}}\left(2 + \frac{2\eta \ln\eta}{1-\eta} \right) = \beta_\eta.
\end{equation*}
\hfill $\blacksquare$

\textbf{Proof of Theorem \ref{thm:joint normality}}
It suffices to show
$$
\sqrt{M(t) - M(t_\eta)} \Delta_{i,t} \Rightarrow \mathcal{N}(0, \tilde{\Sigma}_{i, \infty}), \quad \text{ as } t\rightarrow \infty,$$
where 
\begin{equation*} 
\tilde\Sigma_{i, \infty}(j,j^\prime) :=  \nabla \delta_{ij}(\theta^c)^\intercal \tilde{\Sigma}_D \nabla \delta_{ij^\prime} (\theta^c) + \textbf{Cov}\left( X_i(\theta^c)-X_j( \theta^c),X_i(\theta^c) - X_{j^\prime}(\theta^c)\right),
\end{equation*}
in which $\tilde{\Sigma}_D := diag(\beta_{i,\eta} \Sigma_{D,1}, \ldots, \beta_{S,\eta} \Sigma_{D,S})$ and $\beta$ is defined in (\ref{eq:beta}).
 Let 
$$ \Delta_{i} (\theta)= \left(\delta_{i1}(\theta),\cdots,\delta_{ii-1}(\theta),\delta_{ii+1}(\theta),\cdots,\delta_{iK}(\theta) \right),$$
$$\hat{\Delta}_{i,t} = \left(\hat{\delta}_{i1,t},\cdots,\hat{\delta}_{ii-1,t},\hat{\delta}_{ii+1,t},\cdots,\hat{\delta}_{iK,t} \right),$$
$$\bar{\Delta}_{i,t} =  \left(\bar{\delta}_{i1,t},\cdots,\bar{\delta}_{ii-1,t},\bar{\delta}_{ii+1,t},\cdots,\bar{\delta}_{iK,t} \right),$$
where 
$$  \bar{\delta}_{ij,t} :=  [M(t) - M(t_\eta)]^{-1} \sum_{\ell=t_\eta+1}^t m(\ell) [\mu_i(\hat{\theta}_\ell) - \mu_j(\hat{\theta}_\ell)].
$$
We have the decomposition,
\begin{align*}
&[M(t) - M(t_\eta)]^{1/2} [\hat{\Delta}_{i,t} - \Delta_i(\theta^c)] \\
 =& \underbrace{[M(t) - M(t_\eta)]^{1/2} (\hat{\Delta}_{i,t} - \bar{\Delta}_{i,t})}_{\textstyle \mathstrut Y_t} + \underbrace{[M(t) - M(t_\eta)]^{1/2}[\bar{\Delta}_{i,t} - \Delta_i(\theta^c)]}_{\textstyle \mathstrut Z_t} \\
\end{align*}
It suffices to show the respective convergence of $Z_t$ and $Y_t \mid \mathcal{F}_t$, where $F_t := \sigma(\hat{\theta}_1,\cdots,\hat{\theta}_t)$,  the  filtration  generated  by  input  parameter  estimates. We will show convergence of $Y_t+ Z_t$ through its characteristic function \begin{equation*}
\Phi_t(x) = \E[e^{\mathbf{i}x^T(Y_t + Z_t)}] = \E\left[e^{\mathbf{i}x^T Z_t} \E[e^{\mathbf{i}x^T Y_t} \mid \mathcal{F}_t] \right].
\end{equation*}
If we can show the respective convergence of $Z_t $ and $Y_t \mid \mathcal{F}_t$, then the result follows from the Dominated Convergence Theorem for complex-valued random variables.

\textbf{[(1)]} Convergence of $Y_t \mid \mathcal{F}_t$:
Expand $Y_t$ to get
\begin{align*}
    Y_t =& [M(t) - M(t_\eta)]^{1/2}  (\hat{\Delta}_{i,t} - \bar{\Delta}_{i,t}) \\
    =& [M(t) - M(t_\eta)]^{-1/2} \sum_{\ell=t_\eta+1}^t \sum_{r=1}^{m(t)}\left(
    \begin{tabular}{c}
          $X_{i,r}(\hat{\theta}_\ell) - \mu_i(\hat{\theta}_\ell)-(X_{1,r}(\hat{\theta}_\ell) - \mu_1(\hat{\theta}_\ell))$\\ 
         \vdots\\ 
         $X_{i,r}(\hat{\theta}_\ell) - \mu_i(\hat{\theta}_\ell)-(X_{i-1,r}(\hat{\theta}_\ell) - \mu_{i-1}(\hat{\theta}_\ell))$\\
         $X_{i,r}(\hat{\theta}_\ell) - \mu_i(\hat{\theta}_\ell)-(X_{i+1,r}(\hat{\theta}_\ell) - \mu_{i+1}(\hat{\theta}_\ell))$\\
         \vdots\\
         $X_{i,r}(\hat{\theta}_\ell) - \mu_i(\hat{\theta}_\ell)-(X_{K,r}(\hat{\theta}_\ell) - \mu_{K}(\hat{\theta}_\ell))$
    \end{tabular}
     \right) 
\end{align*}
where, conditioned on $\mathcal{F}_t$, $(X_{i,r}(\hat{\theta}_\ell) - \mu_i(\hat{\theta_\ell} -(X_{j,r}(\hat{\theta}_\ell) - \mu_j(\hat{\theta}_\ell)) _{j\neq i}$ are independent random vectors with mean $0$ and covariance matrix $\operatorname{\Sigma}_i(\hat{\theta_\ell})$. Since $\operatorname{\Sigma}_{X,i}(\hat{\theta_\ell}) \rightarrow \operatorname{\Sigma}_{X,i}({\theta^c})$ almost surely by the continuity of $\Sigma(\cdot)$. we can apply Lemma \ref{lemma:CLT} to get (conditioned on $\mathcal{F}_t$)

$$[M(t) - M(t_\eta)]^{-1/2} \sum_{\ell=t_\eta+1}^t \sum_{r=1}^{m(t)}\left(
    \begin{tabular}{c}
          $X_{i,r}(\hat{\theta}_\ell) - \mu_i(\hat{\theta}_\ell)-(X_{1,r}(\hat{\theta}_\ell) - \mu_1(\hat{\theta}_\ell))$\\ 
         \vdots\\ 
         $X_{i,r}(\hat{\theta}_\ell) - \mu_i(\hat{\theta}_\ell)-(X_{i-1,r}(\hat{\theta}_\ell) - \mu_{i-1}(\hat{\theta}_\ell))$\\
         $X_{i,r}(\hat{\theta}_\ell) - \mu_i(\hat{\theta}_\ell)-(X_{i+1,r}(\hat{\theta}_\ell) - \mu_{i+1}(\hat{\theta}_\ell))$\\
         \vdots\\
         $X_{i,r}(\hat{\theta}_\ell) - \mu_i(\hat{\theta}_\ell)-(X_{K,r}(\hat{\theta}_\ell) - \mu_{K}(\hat{\theta}_\ell))$
    \end{tabular}
     \right)    \Rightarrow \mathcal{N}(0,\Sigma_{X,i}(\theta^c)) \text{ as } t \rightarrow \infty$$

\textbf{[(2)]} Convergence of $Z_t$:\\
Since $\mu_j(\theta)$ is twice continuously differentiable for each $j$, so is $\Delta_i (\theta)$. By Taylor expansion,

\begin{align}
Z_t =& [M(t) - M(t_\eta)]^{-1/2} \sum_{\ell=t_\eta+1}^t m(\ell)[{\Delta_i}(\hat{\theta}_\ell) - \Delta_i(\theta^c)] \notag\\
=&[M(t) - M(t_\eta)]^{-1/2} \nabla (\Delta_i(\theta^c)) \sum_{\ell=t_\eta+1}^t m(\ell) (\hat{\theta}_\ell - \theta^c ) \label{first-term} \tag{$Z_{t,1}$} \\
&+ [M(t) - M(t_\eta)]^{-1/2} \sum_{\ell=t_\eta+1}^t m(\ell) ^\intercal \mathcal{O}  (||\hat{\theta}_\ell - \theta^c||_2 ) (\hat{\theta}_\ell - \theta^c ). \label{second-term} \tag{$Z_{t,2}$}
\end{align}

We show the convergence of $Z_{t,1}$ and $Z_{t,2}$, respectively. 
\textbf{[(2.1)]}
Convergence of $Z_{t,1}$

Our goal is to show that
\begin{equation*}
[M(t) - M(t_\eta)]^{-1/2}  \sum_{\ell=t_\eta+1}^t m(\ell) (\hat{\theta}_\ell - \theta^c) \Rightarrow \mathcal{N}(0, \tilde{\Sigma}), \quad \text{as } t \rightarrow \infty.
\end{equation*}
Since the input distributions are mutually independent, it suffices to show that for any $s$,
\begin{equation} \label{eq:q-conv}
[M(t) - M(t_\eta)]^{-1/2}  \sum_{\ell=t_\eta+1}^t m(\ell) (\hat{\theta}_{s,\ell} - \theta^c_s) \Rightarrow \mathcal{N}(0, \beta_{s, \eta} \Sigma_{D,s}), \quad \text{as } t \rightarrow \infty.
\end{equation}
Note that $\{\hat{\theta}_{s,\ell}\}_{\ell=1}^t$ are linear combinations of $\{D_{s,j}\}_{j=1}^{N_s(t)}$, so we can rearrange the terms to obtain
\begin{align*}
 \sum_{\ell=t_\eta+1}^t m(\ell) (\hat{\theta}_{s,\ell} - \theta^c_s) &= 
 \phi_{i,s,t}(t_\eta+1) \sum_{j=1}^{N_s(t_\eta)} D_{s,j} +\sum_{\ell=t_\eta+1}^t \phi_{i,s,t}(\ell) \sum_{j=N_s(\ell-1)+1}^{N_s(\ell)} D_{s,j},
\end{align*}
where $\phi$ is defined in (\ref{eq:phi}). Since $\{D_{s,j}\}$ are i.i.d. and \ref{conv-cond} holds by assumption, (\ref{eq:q-conv}) can be verified by invoking Lemma \ref{lemma:Lindberg} (details omitted).

Hence, 
$$Z_{t,1} \Rightarrow \mathcal{N}\left(0, \nabla (\Delta_{i,t}(\theta^c))^\intercal\beta_{s, \eta} \Sigma_{D,s}\nabla (\Delta_{i,t}(\theta^c))\right)  \text{ as } t \rightarrow \infty.$$
\textbf{[(2.2)]}
Convergence of $Z_{t,2}$

It is left to show that the residual term $Z_{t,2}$ vanishes in probability. Since a.s.,
\begin{equation*}
 \|\hat{\theta}_\ell - \theta^c \|_2^2 = \sum_{s=1}^S \|\hat{\theta}_{s,\ell} - \theta^c_s \|_2^2 \le M_2 \sum_{s=1}^S \frac{\log\log N_s(\ell)}{N_s(\ell)} \le M_2 S\log (U\ell)/\ell,
\end{equation*}
for some $M_2>0, U\ge \sup_\ell \{\sup_s n_s(\ell),m(\ell)\}$ by the Law of the Iterated Logarithm. Then we have a.s. $\exists M_1 >0$,
\begin{align*}
||Z_{t,2}||_2 & \le [M(t) - M(t_\eta)]^{-1/2}  \sum_{\ell=t_\eta+1}^t m(\ell) M_1 ||\hat{\theta}_\ell - \theta^c||_2^2\\
&\le [M(t) - M(t_\eta)]^{-1/2}  \sum_{\ell=t_\eta+1}^t m(\ell) M_1 M_2 S \log(U\ell)/ \ell \\
&\le (t-t_\eta)^{-1/2} UM_1 M_2 S \log(Ut) \sum_{\ell=t_\eta+1}^t 1/\ell \\
&\le (t-t_\eta)^{-1/2} UM_1 M_2 S \log(Ut) \log(t),
\end{align*}
which converges to 0 as $t\rightarrow \infty$. Hence we get
$$Z_{t} \Rightarrow \mathcal{N}\left(0, \nabla (\Delta_{i,t}(\theta^c))^\intercal\beta_{s, \eta} \Sigma_{D,s}\nabla (\Delta_{i,t}(\theta^c))\right)  \text{ as } t \rightarrow \infty.$$
Together with the convergence of $Y_t | \mathcal{F}_t$, we complete the proof.

\vspace{10pt}
\textbf{Proof of Proposition \ref{prop:opt-eta}}
The proof is developed in three steps.

{\bf 1. Strict convexity of $\psi_\kappa$.}  The first-order and second-order derivatives of $\psi_\kappa$ are given by
\begin{align} 
\psi'_\kappa(\eta) &= \frac{(1+\kappa)(1-\eta) + (1+\eta)\ln \eta}{(1-\eta)^3}, \label{eq:psi-derivative} \\
\psi''_\kappa(\eta) &=  \frac{2(1+\kappa)(1-\eta) + (4+2\eta) \ln\eta - \eta+\eta^{-1}}{(1-\eta)^4}. \notag
\end{align}
We claim that $\psi''_\kappa(\eta) > 0$ on $(0,1)$. To see this, denote the numerator of $\psi''_\kappa(\eta)$ as
\begin{equation*}
g_\kappa(\eta) := 2(1+\kappa)(1-\eta) + (4+2\eta) \ln\eta - \eta+\eta^{-1}.
\end{equation*}
It can be verified that
\begin{align*}
g'_\kappa(\eta) &= -(1+2\kappa) + 4\eta^{-1} - \eta^{-2} + 2\ln\eta, \quad g''_\kappa(\eta) = 2\eta^{-3} (\eta-1)^2.
\end{align*}
We have $g'_\kappa(1) = 2(1-\kappa) < 0$ and $g''_\kappa(\eta) > 0$ for $\eta \in (0,1)$, which implies that $g'_\kappa(\eta) < 0$ for $\eta \in (0,1)$. Furthermore, since $g_\kappa(1) =0$, we know that $g_\kappa(\eta) > 0$ for $\eta \in (0,1)$ and therefore $\psi_\kappa$ is strictly convex in $(0,1)$. Furthermore, since $\lim_{\eta \rightarrow 0} \psi'_\kappa(\eta) = -\infty$ and $\lim_{\eta \rightarrow 1} \psi'_\kappa(\eta) = 1+\kappa$, $\psi_\kappa$ must always have a unique minimizer in $(0,1)$, which we shall denote by $\eta^*(\kappa)$.

{\bf 2. Limiting behavior of $\eta^*(\kappa)$ as $\kappa$ approaches 0 or 1. } To begin with, according to (\ref{eq:psi-derivative}), $\eta^*(\kappa)$ is the solution to the following equation.
\begin{equation*}
\varphi_\kappa(\eta) := (1+\kappa)(1-\eta) + (\eta+1)\ln \eta = 0.
\end{equation*}
\begin{enumerate}
\item[(i)]
$\eta^*(\kappa) \le \kappa^{-1}.$

Note that $\varphi_\kappa$ is concave since $\varphi''_\kappa(\eta) = \eta^{-1} - \eta^{-2} < 0$ for $\eta \in (0,1)$. Together with the fact that $\lim_{\eta \rightarrow 0} \varphi_\kappa(\eta) = -\infty$ and $\lim_{\eta \rightarrow 1} \varphi_\kappa(\eta) = 1+\kappa$, $\varphi_\kappa$ can cross 0 only once. Thus, for any $\eta' \in (0,1)$ such that $\varphi_\kappa(\eta') \ge 0$, we must have $\eta^*(\kappa) \le \eta'$. Finally, 
\begin{align*}
\varphi_\kappa(\kappa^{-1}) &= \kappa - \kappa^{-1} - (\kappa^{-1} + 1) \ln \kappa\\
& \ge \kappa - \kappa^{-1} - (\kappa^{-1} + 1)  (\kappa-1) = 0,
\end{align*}
where the last inequality follows from the fact that $\ln \eta \le \eta - 1$ for all $\eta > 0$. This implies that $\eta^*(\kappa) \le \kappa^{-1}$.

\item[(ii)]
$\lim_{\kappa \rightarrow 1} \eta^*(\kappa) = 1$.

By a similar argument, for any $\epsilon \in (0,1)$ such that $\varphi_\kappa(1-\epsilon) \le 0$, we would have $\eta^*(\kappa) \ge 1 - \epsilon$. Therefore, it suffices to show that 
\begin{equation} \label{eq:phi-limit}
\liminf_{\kappa \rightarrow 1} \varphi_\kappa(1-\epsilon) \le 0, \quad \forall \epsilon \in (0,1).
\end{equation}
as this would imply that $\liminf_{\kappa\rightarrow 1} \eta^*(\kappa) \ge 1- \epsilon$ for any fixed $\epsilon \in (0,1)$. Note that for any fixed $\epsilon$,
\begin{equation*}
\liminf_{\kappa \rightarrow 1} \varphi_\kappa(1-\epsilon) = 2\epsilon + (2-\epsilon) \ln(1-\epsilon):= h(\epsilon),
\end{equation*}
and note that $h(0) = 0$ and $h(\epsilon) \rightarrow -\infty$ as $\epsilon \rightarrow 1$. Furthermore,
\begin{align*}
h'(\epsilon) &= 2 - \ln(1-\epsilon) - \frac{2-\epsilon}{1-\epsilon} \\
&\le 2 - \left(1 - \frac{1}{1-\epsilon}\right) - \frac{2-\epsilon}{1-\epsilon}  = 0,
\end{align*}
where the last inequality follows from the fact that $\ln x \ge 1 - x^{-1}$ for all $x > 0$. We conclude that $h(\epsilon) \le 0$ for any $\epsilon \in (0,1)$.
\end{enumerate}

{\bf Step 3. Maximum variance reduction rate.} We first show that
\begin{equation} \label{eq:inf-ratio-le}
\inf_{\kappa>1} \frac{\psi_\kappa(\eta^*(\kappa))}{\psi_\kappa(0)} \le \frac{1}{2}.
\end{equation}
Note that $\psi_\kappa(0) = \kappa$ for any $\kappa > 1$. Consider $\psi_\kappa(\kappa^{-k}), k \in \Z^+$, for which we have
\begin{equation*}
\psi_\kappa(\kappa^{-k})= \frac{\kappa^k [\kappa(\kappa^k-1) - k\ln \kappa]}{(\kappa^k-1)^2},
\end{equation*}
where sending $\kappa$ to 1 and applying L'Hospital's Rule gives
\begin{align*}
\inf_{\kappa>1} \frac{\psi_\kappa(\eta^*(\kappa))}{\psi_\kappa(0)}  \le \lim_{\kappa \rightarrow 1} \frac{\psi_\kappa(\kappa^{-k})}{\psi_\kappa(0)} = \frac{1}{2} + \frac{1}{k}.
\end{align*}
This implies (\ref{eq:inf-ratio-le}) as $k$ can be arbitrarily large. To prove the opposite direction of (\ref{eq:inf-ratio-le}), note that for any fixed $\eta \in (0,1)$, $\psi_\kappa(\eta)/\psi_\kappa(0)$ is a decreasing function in $\kappa$. Therefore, 
\begin{equation*}
\inf_{\kappa>1} \frac{\psi_\kappa(\eta^*(\kappa))}{\psi_\kappa(0)} \ge 
\inf_{\kappa>1} \frac{\psi_1(\eta^*(\kappa))}{\psi_1(0)} \ge
 \inf_{\eta \in (0,1)} \psi_1(\eta) = \inf_{\eta \in (0,1)} \left\{\frac{1}{1-\eta} + \frac{\eta \ln\eta}{(1-\eta)^2} \right\} = \frac{1}{2},
\end{equation*}
which completes the proof of (\ref{eq:inf-ratio}). 
\hfill $\blacksquare$

\end{document}